\documentclass[11pt]{article}
\usepackage[a4paper]{geometry}
\usepackage[USenglish]{babel} 
\usepackage[T1]{fontenc}
\usepackage[ansinew]{inputenc}
\usepackage{lmodern} 
\usepackage{graphicx} 
\usepackage{amsmath}
\usepackage{amsthm}
\usepackage{amsfonts}
\usepackage{nicefrac, xfrac}
\usepackage{amssymb}
\usepackage{amsmath}
\usepackage{amsthm}
\usepackage{amsfonts}
\usepackage{graphicx}
\usepackage{floatrow}
\usepackage[all]{xy}
\usepackage{caption}
\usepackage{subcaption}
\usepackage{mathtools}
\usepackage{fdsymbol}
\usepackage{xparse}
\usepackage{enumitem}
\usepackage{faktor}
\usepackage{array}
\usepackage{float}
\usepackage{stackengine}
\usepackage{tikz}
\usetikzlibrary{decorations.pathreplacing,decorations.markings}

\tikzset{
 mid arrow/.style={postaction={decorate,decoration={
        markings,
        mark=at position .57 with {\arrow[scale=2]{stealth}}
      }}},
}

\newtheorem{theorem}{Theorem}[section]
\newtheorem*{theorem*}{Theorem}
\newtheorem*{lem*}{Lemma}

\newcolumntype{M}[1]{>{\centering\arraybackslash}m{#1}}

\newcommand{\fs}[3]{\xymatrix{ {#1} \ar@{-}@/_{1pc}/[r]_{{#2}}& {#3}}}

\newcommand{\fss}[5]{\xymatrix{ {#1} \ar@{-}@/_{0.5pc}/[r]_{{#2}}& {#3} \ar@{-}@/_{0.5pc}/[r]_{{#4}}& {#5}}}

\newcommand{\fsss}[7]{\xymatrix{ {#1} \ar@{-}@/_{0.5pc}/[r]_{{#2}}& {#3} \ar@{-}@/_{0.5pc}/[r]_{{#4}}& {#5} \ar@{-}@/_{0.5pc}/[r]_{{#6}}& {#7}}}

\newcommand{\fssss}[9]{\xymatrix@C=5mm{ {#1} \ar@{-}@/_{1pc}/[r]_{{#2}}& {#3} \ar@{-}@/_{1pc}/[r]_{{#4}}& {#5} \ar@{-}@/_{1pc}/[r]_{{#6}}& {#7} \ar@{-}@/_{1pc}/[r]_{{#8}}& {#9}}}

\newcommand{\sfs}[3]{\raisebox{1.5ex}{\xymatrix@C=3mm{{#1} \ar@{-}@/_{0.9pc}/[r]_{{#2}} & {#3}}}}

\newcommand{\sfsss}[7]{\xymatrix@C=3mm{ {#1} \ar@{-}@/_{0.9pc}/[r]_{{#2}}& {#3} \ar@{-}@/_{0.9pc}/[r]_{{#4}}& {#5} \ar@{-}@/_{0.9pc}/[r]_{{#6}}& {#7}}}

\newcommand{\sfssss}[9]{\xymatrix@C=3mm{ {#1} \ar@{-}@/_{0.9pc}/[r]_{{#2}}& {#3} \ar@{-}@/_{0.9pc}/[r]_{{#4}}& {#5} \ar@{-}@/_{0.9pc}/[r]_{{#6}}& {#7} \ar@{-}@/_{0.9pc}/[r]_{{#8}}& {#9}}}

\makeatletter
\DeclareRobustCommand*{\mfaktor}[3][]
{
   { \mathpalette{\mfaktor@impl@}{{#1}{#2}{#3}} }
}
\newcommand*{\mfaktor@impl@}[2]{\mfaktor@impl#1#2}
\newcommand*{\mfaktor@impl}[4]{
   \settoheight{\faktor@zaehlerhoehe}{\ensuremath{#1#2{#3}}}%
   \settoheight{\faktor@nennerhoehe}{\ensuremath{#1#2{#4}}}%
      \raisebox{-0.35\faktor@zaehlerhoehe}{\ensuremath{#1#2{#3}}}%
      \mkern-4mu\diagdown\mkern-5mu%
      \raisebox{0.35\faktor@nennerhoehe}{\ensuremath{#1#2{#4}}}%
}
\makeatother

\makeatletter
\DeclareRobustCommand*{\sfaktor}[3][]
{
   { \mathpalette{\sfaktor@impl@}{{#1}{#2}{#3}} }
}
\newcommand*{\sfaktor@impl@}[2]{\sfaktor@impl#1#2}
\newcommand*{\sfaktor@impl}[4]{
   \settoheight{\faktor@zaehlerhoehe}{\ensuremath{#1#2{#3}}}%
   \settoheight{\faktor@nennerhoehe}{\ensuremath{#1#2{#4}}}%
      \raisebox{0.35\faktor@zaehlerhoehe}{\ensuremath{#1#2{#3}}}%
      \mkern-1mu\diagup\mkern-4mu%
      \raisebox{-0.35\faktor@nennerhoehe}{\ensuremath{#1#2{#4}}}%
}
\makeatother

\newtheorem{cor}[theorem]{Corollary}
\newtheorem{pro}[theorem]{Proposition}
\newtheorem{lem}[theorem]{Lemma}
\newtheorem*{conj}{The $p$-adic Littlewood Conjecture}

\theoremstyle{definition}
\newtheorem{defn}[theorem]{Definition}

\theoremstyle{remark}
\newtheorem*{remark}{Remark}

\title{A Geometric Interpretation of the $p$-adic Littlewood Conjecture}
\date{September 25, 2018}
\author{John Blackman}

\let\phi\varphi
\let\epsilon\varepsilon
\begin{document}
\maketitle

\begin{abstract}
This paper investigates integer multiplication of continued fractions using geometric structures. In particular, this paper shows that integer multiplication of a continued fraction can be represented by replacing one  triangulation of an orbifold with another triangulation. This method is used to show that eventually periodic continued fractions have partial quotients which have exponential growth when iteratively multiplied by $n$, for $n$ any fixed, natural number. 
\end{abstract}

\tableofcontents

\section{Introduction}

The main aim of this paper is to find a geometric analogue to the integer multiplication of continued fractions. Our work is based upon the link between continued fractions and geodesics intersecting the Farey tesselation $\mathcal{F}$, which was noted by M. Humbert as early as 1916 \cite{MGH}. This connection was famously used by C. Series in \cite{MSCF}, who replaced the usual tessellation of $\mathbb{H}$ by fundamental domains of $SL_2(\mathbb{Z})$ with the Farey complex $\mathcal{F}$ "to clarify the somewhat elusive connection" between the modular surface and continued fractions.  
%The link between the Farey tesselation $\mathcal{F}$ and continued fractions is long established, with M. Humbert noting this connection as early as 1916 \cite{MGH}. In 1923 \cite{EMS}, E. Artin used the link between geodesics on the modular surface and continued fractions, to deduce the existence of a geodesic on the modular surface with dense trajectory. It was not until 1985 \cite{MSCF} that these methods were explicitly combined, when C. Series replaced the usual tessellation of $\mathbb{H}$ by $SL_2(\mathbb{Z})$ with the Farey complex $\mathcal{F}$. In doing so, she was able to make "clearer statements" and provided "rationale for Artin's method" \cite{MSCF}. Series's paper was inspired by R. Moeckel, who had used the Farey tessellation and its link to continued fractions to invesitgate of geodesics on modular surfaces (spaces of the form $\mfaktor{\Phi}{\mathbb{H}}$, where $\Phi$ is an admissible subgroup of $SL_2(\mathbb{Z})$). In this paper, we use a similar approach to R. Moeckel \cite{RM} and C. Series \cite{MSCF} to investigate integer multiplication of continued fractions.
% 
%In this paper, we are motivated by a reformulation of the $p$-adic Littlewood Conjecture and the work R. Moeckel \cite{RM} and C. Series \cite{MSCF}, to find a geometric analogue of integer multiplication of continued fractions.
Our motivation for this paper stems from a reformulation of the $p$-adic Littlewood Conjecture (pLC), which roughly states that for a fixed prime $p$ and any real number $\alpha$, the partial quotients of the continued fraction expansion $\overline{p^n\alpha}$ become unbounded as $n$ tends to infinity. 
%In the typical setting, multiplication is a continuous map between continuous objects, however, since continued fractions are discrete realisations of real numbers, 
However, it is not immediately clear how the continued fraction expansion transforms as we multiply by $p$, and so
naturally the question arises: \par
\vspace{\baselineskip}

 \textit{"How can one multiply a continued fraction by a prime/natural number?"} \par
\vspace{\baselineskip}
 One could construct such a map between continued fractions by taking a continued fraction $\overline\alpha$, recovering the real number $\alpha$, multiplying by a natural number $n$ and then computing the continued fraction expansion of $n\alpha$. However, this algorithm is not very fit for our purposes: it provides little explicit information and when implemented results in errors for $\overline{\alpha}\in\mathbb{R}\setminus\mathbb{Q}$. This is because if $\alpha\in\mathbb{R}\setminus\mathbb{Q}$, then $\overline{\alpha}$ will be infinite and so we must first truncate $\overline{\alpha}$ to compute $\alpha$. This truncation leads to computational errors. 
 
Instead, we wish to create a discrete multiplication map $\overline{n}:\overline{\alpha}\mapsto\overline{n\alpha}$, in which we do not have to worry about truncation. Whilst our algorithm is not easily implementable by a computer, it does provide a lot of explicit information and therefore, is useful in theoretical sense. These discrete multiplication maps were shown to exist by J. Vandehey in \cite{NTM} and an explicit arithmetic construction for $p$ prime, was created by M. Northey in \cite{MMT}.
In this paper, we produce an algorithm to attain such a map, by identifying real numbers with geodesic rays in $\mathbb{H}$ and continued fractions with cutting sequences of these geodesic rays with the Farey complex $\mathcal{F}$.
%Our approach is based on the classical link between cutting sequences and continued fractions, which was first presented by E. Artin in \cite{EMS}. This was later refined by C. Series \cite{MSCF}, who introduced the idea of taking the cutting sequence of a geodesic ray  with respect to the Farey complex $\mathcal{F}$.  
In particular, we show that integer multiplication of a continued fraction by $n$, can be understood as a replacement of one triangulation of the orbifold $\mfaktor{\Gamma_0(n)}{\mathbb{H}}$ with another, which we describe in the following theorem.

\begin{itemize}
\item[] \textbf{Theorem \ref{orbi}.} For every continued fraction $\overline{\alpha}$ and any natural number $n$, there are two canonical triangulations $\faktor{T_{\{1,n\}}}{\sim}$ and $\faktor{T_{\{n,n\}}}{\sim}$, and a geodesic ray $\zeta$ on the orbifold $\mfaktor{\Gamma_0(n)}{\mathbb{H}}$, such that the cutting sequence of $\zeta$ with $\faktor{T_{\{1,n\}}}{\sim}$ corresponds to $\overline{\alpha}$ and the cutting sequence of $\zeta$ with $\faktor{T_{\{n,n\}}}{\sim}$ corresponds to $\overline{n\alpha}$. 
\end{itemize}

We then show that a path on a triangulated orbifold is homotopic to a closed curve if and only if its cutting sequence directly corresponds to an \textit{essentially periodic} continued fractions (see Definition \ref{esp}.(2.)).

\begin{itemize}
\item[] \textbf{Theorem \ref{cc}.} Let $\mathcal{O}$ be a \textit{quotient-triangulated} orbifold. Then any infinite path $\zeta$ on $\mathcal{O}$ is homotopic to a closed curve if and only if the corresponding cutting sequence is \textit{essentially periodic}.
\end{itemize}

By looking at pLC in this geometric setting, we are able to obtain some surprising results pertaining to continued fractions, whilst using relatively simple techniques. The main such results are as follows.

\begin{itemize}
\item[] \textbf{Theorem \ref{conden} and Corollary \ref{connum}.} Let $\overline{\alpha}$ be any strictly periodic continued fraction. Then for any natural number $n$, there are    infinitely many convergent denominators and infinitely many convergent numerators of $\overline{\alpha}$ which are divisible by $n$.

\item[] \textbf{Theorem \ref{evp}.} Let $\overline{\beta}$ be an eventually periodic continued fraction. Then for every natural number $n$ there exists natural numbers $a$ and $k$, and an \textit{essentially periodic} continued fraction $\overline{\alpha}$  such that $\overline{mn^k\beta}=ma+\overline{m\alpha}$, for  $m$ any natural number.
\end{itemize}

Theorem \ref{evp} then allows us to relate the growth of eventually periodic continued fractions to the growth of essentially periodic continued fractions. As a result, we get the following proposition.

\begin{itemize}
\item[] \textbf{Proposition \ref{evpplc}.} Let $\overline{\alpha}$ be an eventually periodic continued fraction. Then $\overline{\alpha}$ has partial quotients which grow exponentially. In particular, every eventually periodic continued fraction satisfies pLC.
\end{itemize}

\noindent
This paper is organised as follows.

Section 2 of this paper is aimed to introduce the premlinary constructions that we will use throughout the paper.  In Section \ref{plc}, we provide both the formal statement of both pLC and a reformulation of pLC. In Section 2.2, we recall the clasical link between continued fractions and cutting sequences of a geodesic ray $\zeta$ on $\mathbb{H}$ with the Farey complex $\mathcal{F}$, which was first introduced by M. Humbert in \cite{MGH}. In Section 2.3, we introduce some novel constructions to show how multiplication of a continued fraction can be viewed as taking the cutting sequence of a geodesic ray with respect to a scaled Farey complex.  We observe that $\Gamma_0(n)$ induces a common tesselation of the Farey complex and the $\frac{1}{n}$-scaled Farey complex $\frac{1}{n}\mathcal{F}$. We use this information to show, that if a continued fraction $\overline{\alpha}$ has a convergent denominator divisible by some natural number $n$, then $\overline{n\alpha}$ contains a partial quotient of size at least $n$ [Proposition \ref{pro2}]. Using the construction of fundamental domains of $\Gamma_0(n)$ introduced by R.S. Kulkarni \cite{AGM} (which we cover as background in $\ref{RSK}$), we describe the $\frac{1}{n}$-scaling of the Farey complex as a change in decoration for a fundamental domain of $\Gamma_0(n)$ and give a theoretical multiplication algorithm for every natural number $n$.

In Section 3, we investigate cutting sequences on the orbifolds pertaining to the quotient space $\mfaktor{\Phi}{\mathbb{H}}$ for $\Phi$ a finite subgroup of $PSL_2(\mathbb{Z})$. For all $\mfaktor{\Gamma_0(n)}{\mathbb{H}}$, we show how the discrete multiplication map is induced by a change in quotient triangulation of $\mfaktor{\Gamma_0(n)}{\mathbb{H}}$ [Theorem \ref{orbi}]. We then show that an infinite path on a triangulated orbifold is homotopic to a closed curve if and only if the corresponding cutting sequence is \textit{essentially periodic} (see definition \ref{esp}.(2.)) [Theorem \ref{cc}]. This is used to show that essentially periodic continued fractions are a closed class under multiplication by  rational numbers  [Corollary \ref{rcc}]. We also use Theorem \ref{cc} to show that for any natural number and any strictly periodic continued fraction, there are infinitely many convergent denominators (and convergent numerators) of the strictly periodic continued fraction which are divisible by this natural number [Theorem \ref{conden} and Corollary \ref{connum}]. We then show that the integer multiplication of an eventually periodic continued fraction is in some way determined by the natural multiplication of an essentially periodic continued fraction [Theorem \ref{evp}]. We also provide an alternative proof to the statement that eventually periodic continued fractions satisfy  pLC, which was first shown in \cite{PDs}. We improve on this result by showing that eventually periodic continued fractions have partial quotients which grow exponentially, when iteratively multiplied by $n$, for some integer $n$ [Proposition \ref{evpplc}].

\textit{Acknowledgements:} This research was funded by a Doctoral ESPRC grant, awarded by Durham University. I would like to especially thank my Ph.D. supervisor, Dr. Anna Felikson, for her continued support and encouragement throughout the project. I would also like to thank Matthew Northey for introducing me to pLC and to the problem of finding a discrete multiplication map, as well as for his helpful discussions surrounding this topic. Finally, I would like to thank Dr. Erez Nesharim for his discussions regarding the $t$-adic Littlewood Conjecture and providing me with more insight into this area.

\section{A Geometric Approach to Integer Multiplication of Continued Fractions}

The aim of this section is to introduce the main constructions that we will use in Section 3. Sections \ref{plc} and \ref{CFCS}  are mostly background, with Section \ref{plc} introducing both the pLC and reformulation of it as motivation of this paper, as well as some classical results in Diophantine approximation, and Section \ref{CFCS} introduces the notion of a cutting sequence and then recalls some classical results of C. Series in $\cite{GMN}$ and $\cite{MSCF}$. In Section \ref{TFC}, we explain how the continued fraction expansion of $n\alpha$, for a real number $\alpha$ and integer $n$, is equivalent to the cutting sequence of some geodesic ray $\zeta$ with the scaled Farey complex $\frac{1}{n}\mathcal{F}$. We then show how multiplication of a continued fraction by an integer $n$ can be represented by replacing one decoration of a fundamental domain of $\Gamma_0(n)$ with another.

%Section \ref{TFC} is predominantly novel, with exception of section \ref{RSK}, which recalls results and construction contained in \cite{AGM} regarding  construction of fundamental domains for $\Gamma_0(n)$.

%In this section we introduce the pLC as motivation for this paper. We then recall some classical results of C. Series in $\cite{GMN}$ and $\cite{MSCF}$. In particular, that the cutting sequence of a geodesic ray $\zeta$, starting at the y-axis and terminating at some point $\alpha\in\mathbb{R}\setminus\{0\}$, with the Farey complex is in bijection with the continued fraction expansion of $\alpha$. We extend these results to produce a continuous multiplication map of continued fractions and using the results of R. S. Kulkarni in \cite{AGM}, we show that integer multiplication of continued fractions can be represented by discrete map between two decorations of a fundamental domain in $\mathbb{H}$.

\subsection{The $p$-adic Littlewood Conjecture}\label{plc}

The $p$-adic Littlewood conjecture (pLC) is a specific case \textit{mixed Littlewood conjecture}, which was first proposed by B. de Mathan and O. Teuli\'e   \cite{PDs} in 2004. The purpose of this conjecture was to gain insight into  
the \textit{Littlewood conjecture}, a problem in Diophantine approximation dating back to the 1930's. However, the $p$-adic Littlewood conjecture has proved very interesting in its own right, with significant progess having been made, but no conclusion. Notably, M. Einsiedler and D. Kleinbock  showed in 2005 \cite{EK}, that the set of counter-examples had Hausdorff dimension zero and D. Badhziahin, Y. Bugeaud, M. Einsiedler and D. Kleinbock showed in 2015 \cite{BBEK}, that all potential counterexamples must be non-recurrent (in fact, even stronger statements regarding the mixed Littlewood Conjecture were made). Progress has also been made regarding the $t$-adic Littlewood Conjecture, an analogue of pLC over function fields. In particular, the $t$-adic Littlewood conjecture has been shown to be false for $\mathbb{F}_3$ by F. Adiceam, E. Nesharim and F. Lunnon in \cite{tLC}, and the paper-folding sequence is given as an explicit counter-example. 

In order to explicitly state pLC, we first define the \textit{$p$-adic norm}  and the \textit{distance to the nearest integer function}. The $p$-adic norm is the function $|\cdot|_p$ given by  $|x|_p:=~p^{-\nu_p(x)}$, where $\nu_p(x):=\max\{n\in\mathbb{N}\cup\{0\}:p^n|x\}$  and the distance to the nearest integer is the function $\|\cdot\|$ given by $\|x\|:=\min\{|x-n|:n\in\mathbb{Z}\}$.

Then the statement of the $p$-adic Littlewood conjecture is as follows.

\begin{conj} For every ${\alpha\in\mathbb{R}}$  and $p$ prime, we have:

$$ \liminf\limits_{q\rightarrow\infty} q\cdot|q|_p\cdot\|q\alpha\|=0 $$
\end{conj}

\noindent
In other words, for every $\alpha\in\mathbb{R}$ and $p$ prime, we can find an infinite subsequence $\{q_k\}_{k\in\mathbb{N}}$ such that:
$$\lim\limits_{k\rightarrow\infty} q_k\cdot|q_k|_p\cdot\|q_k\alpha\|=0$$

In order to remove trivial solutions to pLC, we define the $\textit{badly approximables}$ as the set of real numbers $\textit{Bad}:=\{\alpha\in\mathbb{R}:\liminf_{q\rightarrow{\infty}} q\cdot\|q\alpha\|>0\}$. It follows from the fact $|x|_p\leq{1}$ for any $x\in\mathbb{R}$ and the definition of $Bad$, that if $\alpha\not\in{Bad}$, then $\alpha$ satisfies pLC. It follows from this, that if $\alpha$ is a counter example to pLC, then necessarily $\alpha\in\textit{Bad}$. A useful question which arises from this notion of \textit{Bad} is:\par\vspace{\baselineskip}

 "For $\alpha\in\textit{Bad}$, how can one find a subsequence $\{q_k\}_{k\in\mathbb{N}}$ which minimises the $\|q_k\alpha\|$ term?"\par\vspace{\baselineskip}

This question has been very well studied from the point of view of the Littlewood conjecture and Diophantine approximation. In particular, it is well known that the \textit{convergents}  of the \textit{continued fraction expansion} of $\alpha$ gives the best rational approximation of $\alpha$, and thus the \textit{convergent denominators} of the \textit{continued fraction expansion} of $\alpha$ minimise the term $\|q_k\alpha\|$. We define continued fractions and convergent denominators below.

\subsubsection{Continued Fractions and Convergents}
\begin{defn}[Section 10, G.H. Hardy and E.M. Wright \cite{HW}]
A \textit{continued fraction} $\overline{\alpha}$ is an expression of the form 

$$\overline{\alpha}:=a_0+\cfrac{1}{a_1+\cfrac{1}{^{\ddots}+\cfrac{1}{a_r}}}$$

\noindent
where, $a_0\in\mathbb{Z}$ and $a_i\in\mathbb{N}$ for $i\geq1$.
\end{defn}

We will usually write continued fractions as a sequence of $a_i$'s, $\overline{\alpha}=[a_0;a_1,\ldots,a_r]$ and refer to the $a_i$'s as \textit{partial quotients}. The sequence of partial quotients can be either finite or infinite and refer to the corresponding continued fraction as finite or infinite accordingly. Evaluating the continued fraction expression gives a real number $\alpha$, and for any real number $\alpha$ we can find an associated continued fraction expansion. For any $\alpha\in\mathbb{R}$, the continued fraction $\overline{\alpha}$ is finite if and only if $\alpha\in\mathbb{Q}$. For $\alpha\in\mathbb{Q}$, there are two different continued fraction expansions, $[a_0;a_1,\ldots,a_r]$ and $[a_0;a_1,\ldots,a_r-1,1]$, where $a_r>1$. For $\alpha\in\mathbb{R}\setminus\mathbb{Q}$, there is a unique infinite continued fraction expansion.

\begin{defn}[Section 10.2, G.H. Hardy and E.M. Wright \cite{HW}] Let $\overline\alpha=[a_0;a_1,a_2,\ldots]$ be a continued fraction. We define the \textit{$k$-th convergent} of $\overline{\alpha}$ to be $\frac{p_k}{q_k}:=[a_0;a_1,\ldots,a_k]$. 
We can define this iteratively where:
\begin{align*} p_{-1} &= 1   &p_0 &= a_0  &p_k &= a_kp_{k-1}+p_{k-2} \\
q_{-1} &= 0  &q_0 &= 1  &q_k &= a_kq_{k-1}+q_{k-2} 
\end{align*}
 
We refer to the term $p_k$ as the \textit{$k$-th convergent numerator} of $\alpha$ and $q_k$ as the \textit{$k$-th convergent denominator}.
\end{defn}

\subsubsection{Reformulating pLC in terms of Continued Fractions}

Using the notions of continued fractions and convergents of these continued fractions, we can rephrase pLC as a condition on continued fraction expansions. We define the height of $\alpha$, $B(\alpha)$ to be the largest partial quotient in the continued fraction expansion of $\alpha$ (excluding the first partial quotient). In other words,
 $$B(\alpha):=sup\{a_i:\overline{\alpha}=[a_0;a_1,\ldots], i\in\mathbb{N}\}$$
  We then define $\textit{Bad}_{CF}$ to be the set of real numbers with bounded partial quotients i.e.  $\textit{Bad}_{CF}:=\{\alpha\in\mathbb{R}:B(\alpha)<\infty\}$. Using these definitions, we get the following classical lemma.

\begin{lem}
For any $\alpha\in\mathbb{R}$, $\alpha\in\textit{Bad}$ if and only if $\alpha\in\textit{Bad}_{CF}$. In particular, $\textit{Bad}\equiv\textit{Bad}_{CF}$.
\end{lem}

Since the convergent denominators $\{q^{(n)}_k\}_{k\in\mathbb{N}}$ of $p^n\alpha$ are "good" approximations for $\|q_k(p^n\alpha)\|$, that $\{p^n\cdot{q^{(n)}_k\}}_{k\in\mathbb{N}}$ are also "good" approximations of $\|q_k\alpha\|$. Using a more formal version of this reasoning, we can recover the following reformulation of pLC.

\begin{pro} Let $\alpha\in\textit{Bad}$, then $\alpha$ satisfies pLC if and only if:
$$  \limsup_{i \rightarrow\infty} B(p^i\alpha)=\infty 
$$
\end{pro}

\begin{proof} See \cite{MMT}, Appendix.
\end{proof}

%This condition is a well known fact in Diophantine approximation, which was first brought to the  attention of the author by M. Northey and D. Badhiazin. The author is unaware of the origin of this statement or a reference, so a proof is included in Appendix 1 for completeness.

\begin{remark} It is worth noting that our definition of $B(\alpha)$ excludes the "$a_0$" term. This is due to the fact that, if we were to include this term, then $\limsup_{n \rightarrow{\infty}} B(p^n\alpha)=\infty$ for every $\alpha\in\mathbb{R}$, since  $a^{(k)}_0={\lfloor{p^k\alpha}\rfloor}\rightarrow\infty$, where $\overline{p^k\alpha}=[{a^{(k)}_0;a^{(k)}_1,\ldots}]$. \end{remark}

Since the latter formulation of pLC is a condition on continued fractions, it would be useful to construct a way of computing $\overline{p^n\alpha}$ from the continued fraction $\overline\alpha$. For every $\alpha\in\mathbb{R}$, we can construct a bijective map between $\alpha$ and $\overline\alpha$ (if $\alpha\in\mathbb{Q}$, take $\overline\alpha$ with 1 as the final partial quotient) and a bijective map between $\alpha$ and $p\alpha$. Thus we should be able to construct the bijective map $\overline{p}:\overline{\alpha}\rightarrow\overline{p\alpha}$ to get the following commutative diagram:

\begin{center}
\scalebox{1.2}{
    \xymatrix{
        \alpha \ar[r]^p \ar[d]_{} &  p\alpha \ar[d]^{} \\
        \overline\alpha  \ar@{.>}[r]_{\overline{p}}       & \overline{p\alpha}
}}
\end{center}

We can view continued fractions as discrete realisations of continuous objects and so, we can think of the $\overline{p}$ map as a map between discrete structures. Therefore, the reformulation of pLC produces our motivating question:\par\vspace{\baselineskip}
"Can we construct the $\overline{p}$ map to directly compute $\overline{p\alpha}$ from $\overline{\alpha}$?"
\par\vspace{\baselineskip}

In our setting we will replace $p$, prime, with $n$, a natural number, and $\overline{p}$ with $\overline{n}$ analogously. 

\subsection{Continued Fractions as Cutting Sequences}\label{CFCS}

In this section we will introduce the notion of cutting sequences of both geodesic rays and paths with an ideal triangulation in $\mathbb{H}$, and recall some of the main results of C. Series in \cite{GMN} and \cite{MSCF}. 
 
%In this section we will introduce the notion of cutting sequences for both paths and geodesic rays. We then discuss the relation between the continued fraction expansion of $\alpha\in\mathbb{R}$ and the cutting sequence of geodesic rays $\zeta_\alpha$, starting at the $y$-axis and terminating at $\alpha$, with the Farey complex $\mathcal{F}$ \cite{MSCF}. Finally, we show how the cutting sequence of $\zeta_\alpha$ with the "$\frac{1}{n}$-scaled Farey complex" $\frac{1}{n}\mathcal{F}$ relates to the continued fraction expansion of $n\alpha$.

\subsubsection{Cutting Sequences of Geodesic Rays}

In this paper we will take $\mathbb{H}$ to be the upper half plane $\{z\in\mathbb{C}\cup\{\infty\}:Im(z)\geq0\}$ with boundary $\partial\mathbb{H}=\mathbb{R}\cup\{\infty\}$. Geodesic lines are given by Euclidean half-lines of the form $\{a+iy:0\leq{y}\leq\infty\}$ and semicircles centred on $\partial{\mathbb{H}}$. We define a hyperbolic $n$-gon with vertices $z_1,z_2,\ldots,z_n\in\mathbb{H}$, to be the closed region bounded by $l_1,\ldots,l_n$, where $l_i$ is the geodesic segment between $z_i$ and $z_{i+1}$ (taking $z_{n+1}=z_1$). An \textit{ideal triangle} is a hyperbolic $3$-gon, with all vertices lying on $\partial\mathbb{H}$ and an \textit{ideal triangulation} of $\mathbb{H}$ is an infinite collection of ideal triangles $T$ such that the closure of these triangles cover $\mathbb{H}$ and for any two triangles $\tau_1,\tau_2$ in $T$, $\tau_1\cap\tau_2=\emptyset$. See Fig.~\ref{LT} (in Section \ref{csp}).

Let $\zeta$ be an oriented geodesic, which enters a triangle $\bigtriangleup{ABC}$, labelled clockwise, through the edge $AB$. We define the triangle to be a \textit{left triangle} for $\zeta$ if the geodesic leaves through the edge $BC$ or  a \textit{right triangle} if the geodesic leaves through the edge $AC$. If the geodesic, instead leaves through the vertex $C$, we can view the triangle as either a left triangle or a right triangle. We refer to multiple left triangles in a row as \textit{left fans} and multiple right triangles in a row as \textit{right fans}. For either type of fan, all triangles in the fan will have a common vertex. See Fig.~\ref{fans}.

\begin{figure}[htbp]
        \centering
        \begin{subfigure}[b]{0.45\textwidth}
            \centering
            \includegraphics[width=\textwidth]{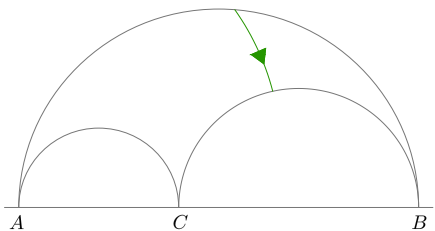}
            \caption{{An example of a left triangle.}}    
            \label{lt}
        \end{subfigure}
        \quad
        \begin{subfigure}[b]{0.45\textwidth}  
            \centering 
            \includegraphics[width=\textwidth]{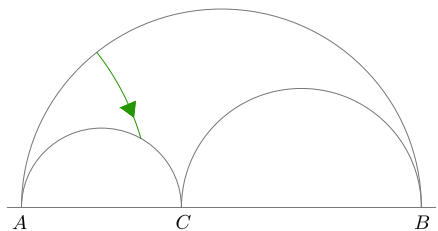}
            \caption{{An example of a right triangle.}}    
            \label{rt}
        \end{subfigure}
        \vskip\baselineskip
        \begin{subfigure}[b]{0.45\textwidth}   
            \centering 
            \includegraphics[width=\textwidth]{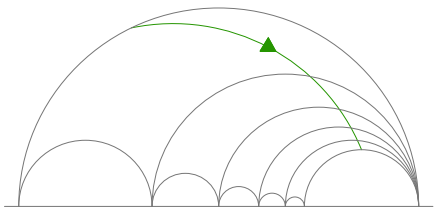}
            \caption{{An example of a left fan.}}    
            \label{lf}
        \end{subfigure}
        \quad
        \begin{subfigure}[b]{0.45\textwidth}   
            \centering 
            \includegraphics[width=\textwidth]{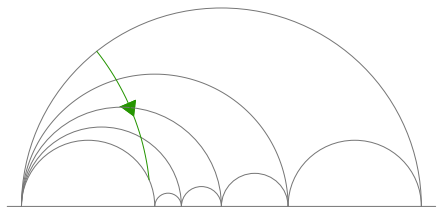}
            \caption{{An example of a right fan.}}    
            \label{rf}
        \end{subfigure}
        \caption{Examples of left and right triangles and fans.} 
        \label{fans}
    \end{figure}

%We define an \textit{ideal triangle} in $\mathbb{H}$ to be a triangle in $\mathbb{H}$ with all vertices lying on the boundary $\partial\mathbb{H}$. An \textit{ideal triangulation} of $\mathbb{H}$ is an infinite collection of ideal triangles, such that every point $x\in\mathbb{H}$ is either a vertex of infinitely many triangles, a point along an edge of exactly two ideal triangles or a point inside a face of exactly one triangle.

\begin{defn}
Let $T$ be an ideal triangulation of $\mathbb{H}$, let $E$ be any edge of $T$ and let $\zeta$ be an oriented geodesic ray starting at $E$ and terminating at some point in $\partial{\mathbb{H}}$. The $\textit{cutting sequence}$ of $\zeta$ with respect to $T$, denoted $(\zeta,T)$, is the potentially infinite word in the alphabet $F_2=<L,R>$, formed by the following process:
 
\begin{itemize}
\item Start with the (empty) word $L^0$.
\item Whenever $\zeta$ cuts $T$ to form a left triangle, add a letter $L$ to the end of the word.
\item Whenever $\zeta$ cuts $T$ to form a right triangle, add a letter $R$ to the end of the word.
\item Repeat this process iteratively to obtain the cutting sequence $(\zeta,T)\hspace{-2pt}:=\hspace{-2pt}L^{n_0}\hspace{-2pt}R^{n_1}\hspace{-2pt}L^{n_2}\cdots\hspace{-1pt}$, where $n_0\in\mathbb{N}\cup\{0\}$ and $n_i\in\mathbb{N}$.
\end{itemize}

\end{defn}

We will identify the cutting sequence $(\zeta,T)=L^{n_0}R^{n_1}L^{n_2}\cdots$ with the sequence of indices $\{n_0,n_1,n_2,...\}$, where $n_0\in\mathbb{N}\cup\{0\}$ and $n_i\in\mathbb{N}$. Since every cutting sequence is of the form $\{\{n_0,n_1,n_2,\ldots\}:n_0\in\mathbb{N}\cup\{0\}, n_i\in\mathbb{N}\quad \forall{i}\in\mathbb{N}\}$, there is an obvious bijection between cutting sequences of the above form and the continued fraction expansion of some $\alpha\in\mathbb{R}_{>0}$. Explicitly, we can take the following bijection $\eta:=\{n_0,n_1,n_2,\ldots\}\mapsto[n_0;n_1,n_2,\ldots]$.

%Let $T$ be an ideal triangulation of $\mathbb{H}$ and let $E$ be any edge of this triangulation. If we take any oriented geodesic ray $\zeta_E$, starting at $E$ and terminating at some point $\alpha\in\mathbb{R}\cup\{\infty\}$, $\zeta_E$ will intersect the triangulation $T$ to form a potentially infinite sequence of right and left fans. We encode this data by counting the number of triangles in the initial fan $n_0$ and append this number to an empty sequence. We then take the next fan, count the number of triangles in this fan $n_1$ and again append it to the sequence. We continue this process iteratively, to obtain a potentially infinite sequence of natural numbers $\{n_0,n_1,n_2,...\}$, which we refer to as the \textit{cutting sequence} of $\zeta_E$ with respect to $T$. We will represent cutting sequences as the pair $(\zeta_E,T)$, where $T$ is the choice of triangulation $T$ and $\zeta_E$ is a oriented geodesic ray stemming from some edge $E$. Convention will be to always start with a left fan, though this may be empty, in which case the first term is $0$.  Since  every cutting sequence is written in the form $\{\{n_0,n_1,n_2,\ldots\}:n_0\in\mathbb{N}\cup\{0\}, n_i\in\mathbb{N} \forall{i}\in\mathbb{N}\}$, there is an obvious bijection between cutting sequences of the above form and the continued fraction expansion of some $\alpha\in\mathbb{R}_{>0}$. Explicitly, we can take the following bijection $\eta:\{n_0,n_1,n_2,\ldots\}\mapsto[n_0,n_1,n_2,\ldots]$.

\begin{remark} If a cutting sequence of an oriented geodesic ray $\zeta$ is finite, then the geodesic ray terminates at a vertex of the triangulation. The final triangle which $\zeta$ intersects can be thought of either a left triangle or right triangle and thus, can either be added to the final fan or represent a new fan on its own. This is analogous to the fact that the two finite continued fractions expansions $[a_0;a_1,\ldots,a_r]$ and $[a_0;a_1,\ldots,a_r-1,1]$ are equivalent.
\end{remark}

Occasionally, it may be useful to take the cutting sequence of a geodesic ray $\zeta$ starting at an edge $E$, from an  edge $F$ later cutting sequence of $(\zeta,T)$. We will denote this cutting sequence as $(\zeta,T)_F$. We can think of $(\zeta,T)_F$ as a copy of $(\zeta,T)=(\zeta,T)_E$ with a prefix removed. That is, $(\zeta,T)_F$ coincides with $(\zeta,T)$ except for finitely many terms at the start. It is worth noting that due to the convention of always starting with an $L^0$ term, the types of triangle will also coincide for these terms when written as sequences of indices. 

Every edge $E$ in $T$ separates $\mathbb{H}$ into two regions, which we will arbitrarily label $E_+$ and $E_-$. Similarly, $E$ will separate any geodesic in $\mathbb{H}$, which it intersects transversly, into two disjoint geodesic rays, one contained entirely in $E_+$ and the other contained entirely in $E_-$. As a result, any geodesic ray $\zeta$ starting at $E$ will be contained entirely in one of these two regions. We will denote the set of all geodesic rays starting at $E$, which are contained entirely in $E_+$ as $Z_{E_+}$ and likewise will denote the set of all geodesic rays starting at $E$, which are contained entirely in $E_-$ as $Z_{E_-}$. Given a particular cutting sequence and triangulation $T$, for every edge $E$ in $T$ we can find two distinct classes of geodesic rays which have this cutting sequence: one in $Z_{E_+}$ and the other in $Z_{E_-}$. These classes are completely determined by the endpoint of one such geodesic ray. That is to say, all geodesic rays starting at $E$ and terminating at a fixed point $\alpha\in\mathbb{R}\cup\{\infty\}$ will have the same cutting sequence relative to $T$. The proof of this statement is an analogue of Lemma 3.1.1 from \cite{MSCF}, which we explicitly state in section \ref{FC}.

Due to the fact that, for any geodesic ray $\zeta$ and any ideal triangulation $T$, all orientation preserving isomorphisms of $\mathbb{H}$ ($Isom^+(\mathbb{H})$) preserve both the notions of left and right triangles, and how $\zeta$ and $T$ intersect each other, we get the following lemma.

\begin{lem}
Cutting sequences are invariant under $Isom^+(\mathbb{H})=PSL_2(\mathbb{R})$: If $\phi\in{Isom^+(\mathbb{H})}$, then for any geodesic ray $\zeta$ and any triangulation $T$, $(\zeta,T)=\phi((\zeta,T))=(\phi(\zeta),\phi(T))$. 
\end{lem}

If we take an orientation-reversing automorphism $\psi\in{Isom(\mathbb{H})}\setminus{Isom^+(\mathbb{H})}$, the notions of left and right triangles swap and as such  $(\psi(\zeta),\psi(T))=(-\zeta,T)=(\zeta^{-1},T)$. Here, we take $(\zeta,T)=\{a_0,a_1,\ldots\}$ and assume $\zeta\in{Z_{E_+}}$. Then $\zeta^{-1}$  is a geodesic ray in $Z_{E_+}$ with $(\zeta^{-1},T)=\{0,a_0,a_1,\ldots\}$ and $-\zeta$ is a geodesic ray in $Z_{E_-}$ with $(-\zeta,T)=\{0,a_0,a_1,\ldots\}$.

\subsubsection{Cutting Sequences of Paths}\label{csp}

It will often be useful to deal with paths starting from some edge $E$ and terminating at some point in $\partial\mathbb{H}$, instead of geodesic rays. In particular, it will be useful to see how homotopy affects cutting sequences. To do so, we will extend the definition of cutting sequences to include paths which may double back on themselves. We do this by labelling the sides of the triangles in the triangulation and expressing cutting sequences as powers of these labels.

%We define an \textit{inner labelling} of a triangle $\triangle{ABC}$, labelled clockwise, to be a choice of letters for the inside of every edge in the triangle. We can think of an inner labelling as a map $\phi:\{l_1,l_2,l_3\}\rightarrow\triangle{ABC}$, such that $l_1$ labels the inside of the edge $AB$, $l_2$ labels the inside of the edge $BC$ and $l_3$ labels the inside of the edge $AC$. We define a \textit{labelled triangulation} to be a choice of inner labelling for all triangles in the triangulation. See Fig.~\ref{Lt} for an example of a labelled triangulation.

In order to form a labelled triangulation, we first must create two labelling maps $\phi_1$ and $\phi_2$. Here, we take an arbitrary triangle $\triangle{ABC}$, with vertices labelled clockwise. We then define $\phi_1:\{L^{-1},L,R\}\longrightarrow{\triangle{ABC}}$ to be the labelling map such that $L^{-1}$ labels the inside of the edge $AB$, $L$ labels the inside of the edge $BC$ and $R$ labels the inside of the edge $AC$. Similarly, we define  $\phi_2:\{R^{-1},L,R\}\longrightarrow{\triangle{ABC}}$ to be the labelling map such that $R^{-1}$ labels the inside of the edge $AB$, $L$ labels the inside of the edge $BC$ and $R$ labels the inside of the edge $AC$. See Fig.~\ref{phi}.

\begin{figure}[htbp]
          \centering
        \begin{subfigure}[b]{0.4\textwidth}
            \centering
            \includegraphics[width=\textwidth]{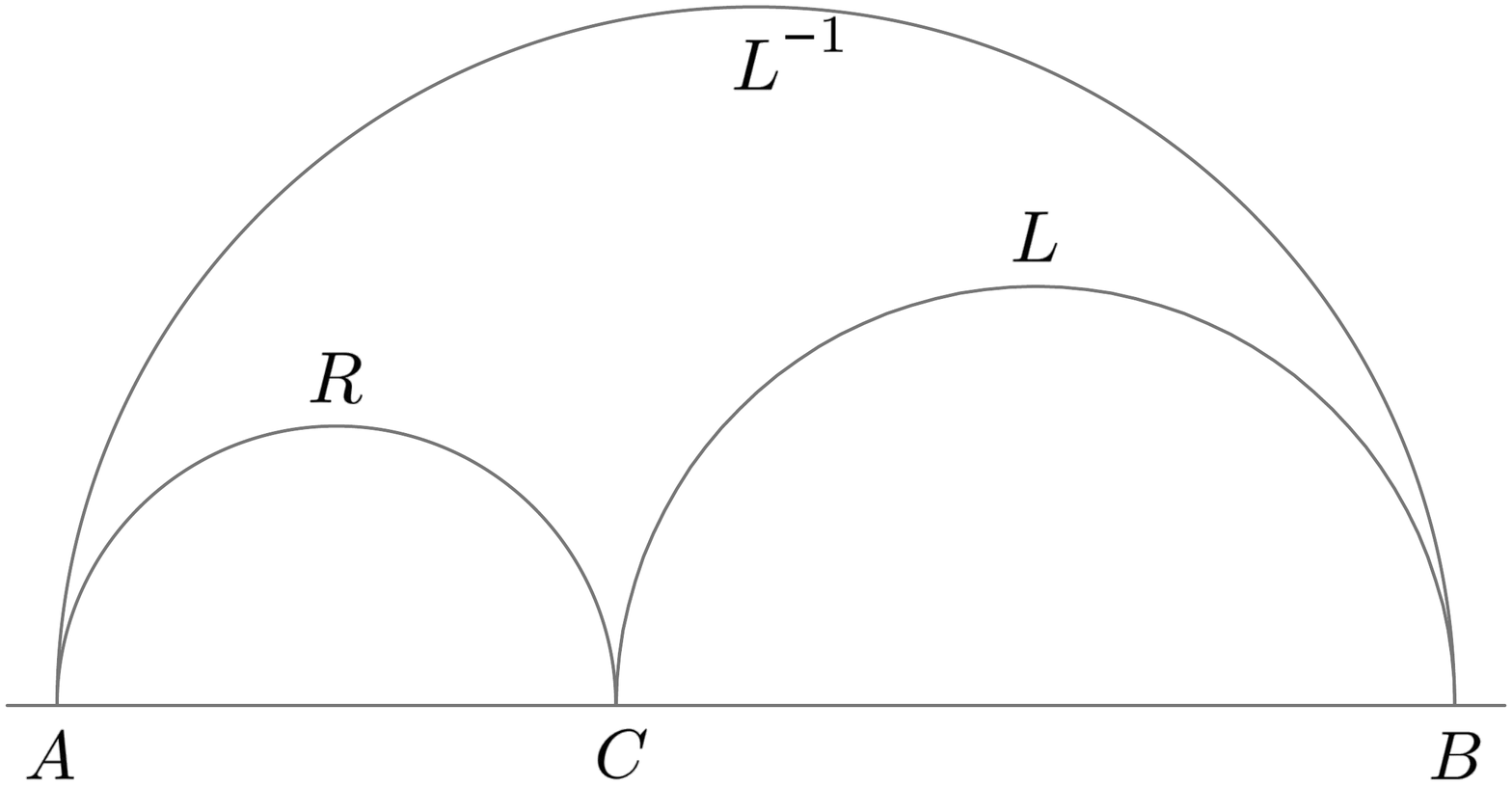}
            \caption{{An example of $\phi_1$ labelling a triangle $\triangle{ABC}$.}}    
            \label{p1}
        \end{subfigure}
        \quad
        \begin{subfigure}[b]{0.4\textwidth}  
            \centering 
            \includegraphics[width=\textwidth]{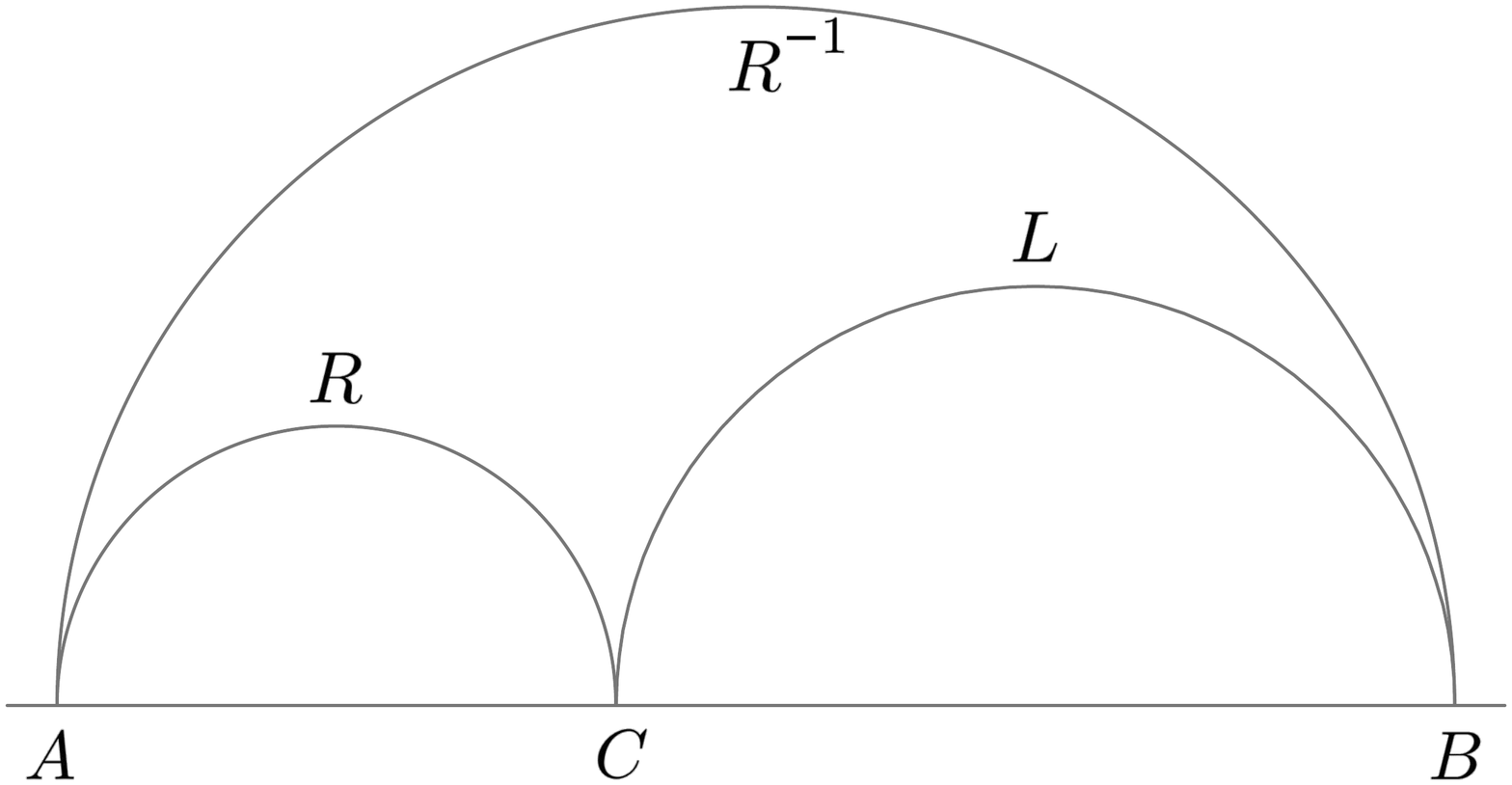}
            \caption{{An example of $\phi_2$ labelling a triangle $\triangle{ABC}$.}}    
            \label{p2}
        \end{subfigure}
        
        \caption{Examples of $\phi_1$ and $\phi_2$ inducing labelling of an arbitrary triangle $\triangle{ABC}$, with vertices labelled clockwise.}\label{phi}
\end{figure}

Let $T$ be an ideal triangulation of $\mathbb{H}$, $E$ be an edge of $T$ and $\lambda$ be an oriented path in $\mathbb{H}$, which starts at $E$, terminates at $\alpha\in{\partial\mathbb{H}}$ and is otherwise disjoint from $\partial{\mathbb{H}}$. The edge $E$ is an edge of exactly two triangles in $T$: $\tau_+$, which is contained in $E_+$ and $\tau_-$, which is contained in $E_-$. The endpoint $\alpha$ of $\lambda$, will lie in either $E_+$ or $E_-$. Without loss of generality, assume that the end point $\alpha$ lies in ${E_+}$. Label $\tau_+$ using $\phi_1$ such that $L^{-1}$ labels the inside of edge $E$. Then the following algorithm produces a labelling for $T$:\par
\vspace{\baselineskip}

Let $\tau$ be a labelled triangle in $T$. Pick an edge $X$ in $\tau$ and let $\tau'$ be the unique other  triangle (unlabelled) in $T$ with edge $X$.
\begin{itemize}
\item If $X$ has inner label $L$ in $\tau$, label $\tau'$ using $\phi_1$ such that $L^{-1}$ is the inner labelling of $X$ in $\tau'$.

\item If $X$ has inner label $R$ in $\tau$, label $\tau'$ using $\phi_2$ such that $R^{-1}$ is the inner labelling of $X$ in $\tau'$.

\item If $X$ has inner label $L^{-1}$ in $\tau$, label $\tau'$ using $\phi_1$ such that $L$ is the inner labelling of $X$ in $\tau'$.

\item If $X$ has inner label $R^{-1}$ in $\tau$, label $\tau'$ using $\phi_2$ such that $R$ is the inner labelling of $X$ in $\tau'$.

\end{itemize}

Repeat \textit{ad infinitum}. See Fig.~3 for an example of a labelled triangulation.\par
\vspace{\baselineskip}

\begin{figure}[htbp]
  \includegraphics[width=0.8\linewidth]{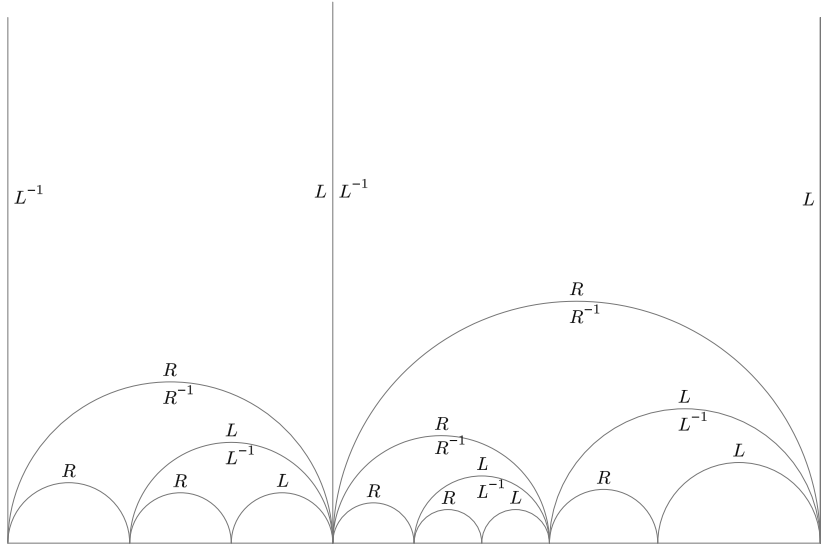} 
  \caption{An example of a labelled triangulation.}
  \label{LT}
\end{figure}

\noindent
The \textit{generalised cutting sequence} of an oriented path $\lambda$ with a labelled triangulation $T$ is defined as follows.

Start with the word $L^0$ in ${F}_2=\langle{L,R}\rangle$. Then every time $\lambda$ passes through an edge of the triangle, append the inner label of that edge (of that triangle) to the word. We define the  \textit{generalised cutting sequence} of $\lambda$ with respect to $T$, which we denote $(\lambda,T)$, to be the word formed by repeating this process iteratively i.e. $(\lambda,T)=L^{m_0}R^{m_1}L^{m_2}\cdots$ with $m_i\in\mathbb{Z}$ $\forall{i\in\mathbb{N}\cup\{0\}}$. Convention will be to always start with a power of $L$ and to have the word alternate between powers of $L$'s and $R$'s, but the index of these letters may be $0$. In particular, a path which passes through an edge and then immediately passes through that edge again would correspond to the term $\cdots{L^{k_1}R^0L^{k_2}}\cdots$ or $\cdots{R^{k_1}L^0R^{k_2}}\cdots$, for some $k_1,k_2\in\mathbb{Z}$. 
As we did for cutting sequences of geodesic rays, we will identify the generalised cutting sequence of a path $(\lambda,T)=L^{m_0}R^{m_1}L^{m_2}\cdots$  with the sequence of indices $\{m_0,m_1,m_2,\ldots\}$, where $m_i\in\mathbb{Z}$ $\forall{i\in\mathbb{N}\cup\{0\}}$.  

\begin{remark}
For $\zeta$ a geodesic ray and an ideal triangulation $T$, the notions of cutting sequence and generalised cutting sequence are equivalent. As a result, we will drop the term "generalised" refer to both as a "cutting sequence".
\end{remark}

We will say that the cutting sequence $(\lambda,T)$ is \textit{reduced} if the word contains no term of the form $gx^0g^{-1}$, where $g\in\langle{L,R}\rangle$ and $x\in\{L,R\}$. We can reduce the cutting sequence by reducing the corresponding word and will denote the class of all equivalent cutting sequences up to reduction as $[\lambda,T]$. It follows quite simply that cutting sequences will be reduced if and only if the corresponding path $\lambda$ does not pass through any edge of $T$ more than once. We can view reduction of the cutting sequence as a homotopy of the path $\lambda$, which preserves $\partial{\mathbb{H}}$. As a result, the classes of equivalent cutting sequences $[\lambda,T]$ are exactly the classes of homotopic paths $[\lambda]^E_\alpha$ with the same starting edge $E$ and endpoint $\alpha$.

%It is worth noting that the corresponding word will be reduced if and only if the path does not pass through the same edge multiple times. We will, without loss of generality, assume that $\alpha$, the endpoint of $\lambda$ lies in the region $E_+$. If the cutting starts with an $L^{-1}$, then $\lambda$ must contain a subpath which lies entirely in $E_-$. In particular, $\lambda$ must pass through $E$ again to enter the region $E_+$. Similar arguments can be used for any edge in $T$ and thus the result follows. As a result, if $\lambda$ is a geodesic ray, it does not double back on itself and hence gives a reduced cutting sequence.
We can analogously define the map between cutting sequences and continued fraction expansions $\eta:\{m_0,m_1,m_2,\ldots\}\mapsto[m_0;m_1,m_2,\ldots]$. Since we take alternating letters, continued fractions reduce in the same way that the corresponding words do. For example, if we had a non-reduced word ${g_1}L^{k_1}R^0L^{-k_1}g_2\cdots={g_1}L^{0}{g_2}\cdots$ (for $g_1,g_2\in\langle{L,R}\rangle$), then the corresponding continued fraction expansion would be $[\widetilde{g_1},k_1,0,-k_1,\widetilde{g_2},\ldots]$, which is equivalent under concatenating terms to $[\widetilde{g_1},k_1+(-k_1),\widetilde{g_2},\ldots]=[\widetilde{g_1},0,\widetilde{g_2},\ldots]$. The above construction ensures that the reduction of a cutting sequence directly corresponds to the reduction of the equivalent continued fraction. See Fig.~\ref{fig:test} for an example of two equivalent cutting sequences.

\begin{figure}[htbp]
\centering
\begin{subfigure}{.8\textwidth}
  \centering
  \includegraphics[width=1\linewidth]{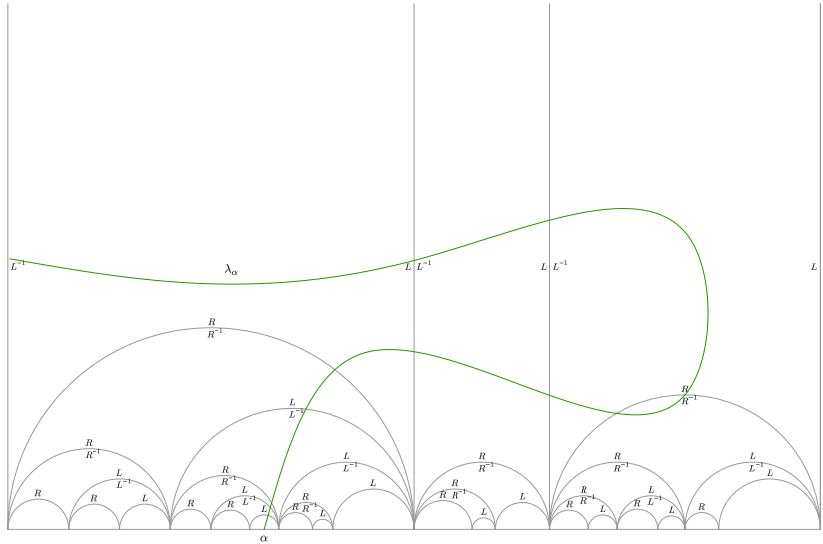}
  \caption{ A path $\lambda_\alpha$ with non-reduced cutting sequence $\{1,1,1,0,-1,1,1,1,\ldots\}$}
  \label{fig:sub1}
\end{subfigure}%
\\
\begin{subfigure}{.8\textwidth}
  \centering
  \includegraphics[width=1\linewidth]{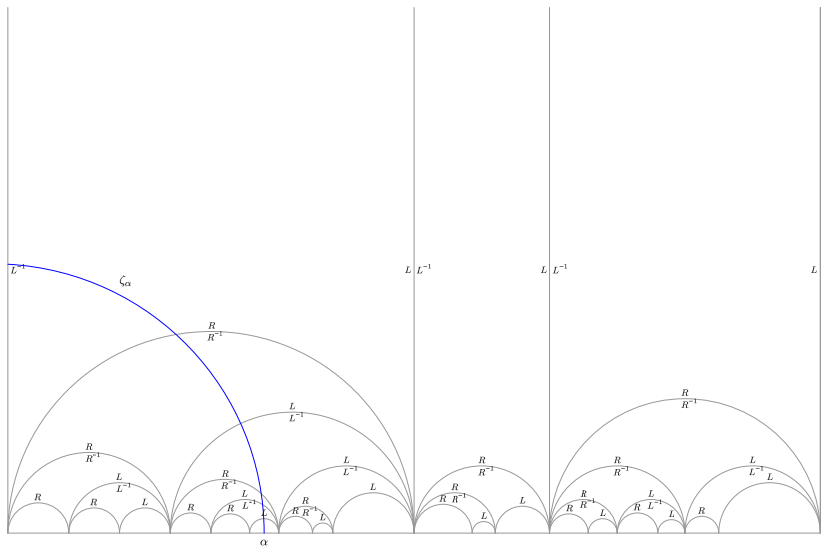}
  \caption{ A geodesic $\zeta_\alpha$ with reduced cutting sequence $\{0;1,2,1,1,\ldots\}$}
  \label{fig:sub2}
\end{subfigure}
\caption{An example of two different paths intersecting a labelled triangulation to form the same cutting sequence up to reduction.}
\label{fig:test}
\end{figure}

\subsubsection{The Farey Complex $\mathcal{F}$}\label{FC}

The Farey complex $\mathcal{F}$ is an ideal triangulation of the upper-half plane $\mathbb{H}$. The vertices are the set $\mathbb{Q}\cup\{\infty\}$. Two vertices $A$ and $B$ have a geodesic edge between them if once written in reduced form, $A=\frac{p}{q}$ and $B=\frac{r}{s}$, we have $\mid{ps-qr}\mid=1$. We say two vertices are neighbours, if they have an edge between them. In this definition, we treat $\infty$ as $\frac{1}{0}$. An equivalent way of interpreting $\mathcal{F}$ is by taking the image of the line between $0$ and $\infty$ under all possible elements of $SL_2(\mathbb{Z})$. See Fig.~\ref{Conv} for a truncated picture of the Farey Complex.

Given two vertices $A=\frac{p}{r}$ and $B=\frac{q}{s}$ in $\mathbb{Q}\cup\{\infty\}$ in reduced form, we can define \textit{Farey addition} $\oplus$ and \textit{Farey subtraction} $\ominus$, as follows:

\begin{center}

$A\oplus{B} := \frac{p+r}{q+s}=\frac{r+p}{s+q}=:B\oplus{A} $\\

$A\ominus{B} := \frac{p-r}{q-s}=\frac{r-p}{s-q}=:B\ominus{A}$

\end{center} 

Simple arithmetic can show that any two neighbours $A$ and $B$ in $\mathcal{F}$ have exactly two neighbours in common, $A\oplus{B}$ and $A\ominus{B}$. It is a well known fact, that given any two neighbours in $\mathcal{F}$, you can generate the whole of $\mathcal{F}$ by iteratively using Farey addition and subtraction of these two points. 

The following theorem, highlights the importance of the Farey Complex with regards to continued fractions.

\begin{theorem}[Theorem A, C. Series \cite{MSCF}]\label{thma} Let $\zeta$ be a geodesic in $\mathbb{H}$ with  endpoints $\alpha_1>0$ and $\alpha_2<0$, and let $I$ be the geodesic line between $0$ and $\infty$, $I_+$ be the region $\{z:Re(z)>0\}$ and $I_-$ be the region $\{z:Re(z)<0\}$. Then, for $\zeta^+=\zeta\cap{Z_{I_+}}$ and $\zeta^-=\zeta\cap{Z_{I_-}}$ (with implicit orientation), $\eta((\zeta^+,\mathcal{F}))$ is the continued fraction expansion of $\alpha_1$ and $\eta((\zeta^-,\mathcal{F}))$ is the continued fraction expansion of $\frac{-1}{\alpha_2}$.
\end{theorem}

An immediate consequence of this theorem is that if we take $Z_\alpha$ to be the set of all geodesic rays starting at the y-axis $I$ and terminating at a fixed point $\alpha\in\mathbb{R}\setminus\{0\}$, then $(\zeta,\mathcal{F})=(\zeta',\mathcal{F})$ for all $\zeta,\zeta'\in{Z_\alpha}$. This result can be extended to Lemma 3.3.1 from \cite{MSCF}:

\begin{lem}[Lemma 3.1.1, C. Series \cite{MSCF}] Let $\zeta_{\alpha,1}$ and $\zeta_{\alpha,2}$ be two geodesic rays with the same endpoint $\alpha\in\mathbb{R}\setminus\mathbb{Q}$ but potentially different start points. Then the cutting sequences of $\zeta_{\alpha,1}$ and $\zeta_{\alpha,2}$ with respect to $\mathcal{F}$, eventually coincide. In particular, there exists an edge $E$ which both $\zeta_{\alpha,1}$ and $\zeta_{\alpha,2}$ intersect, such that $(\zeta_{\alpha,1},\mathcal{F})_{E}=(\zeta_{\alpha,2},\mathcal{F})_{E}$.
\end{lem}

\begin{figure}[htbp]
  \includegraphics[width=0.8\linewidth]{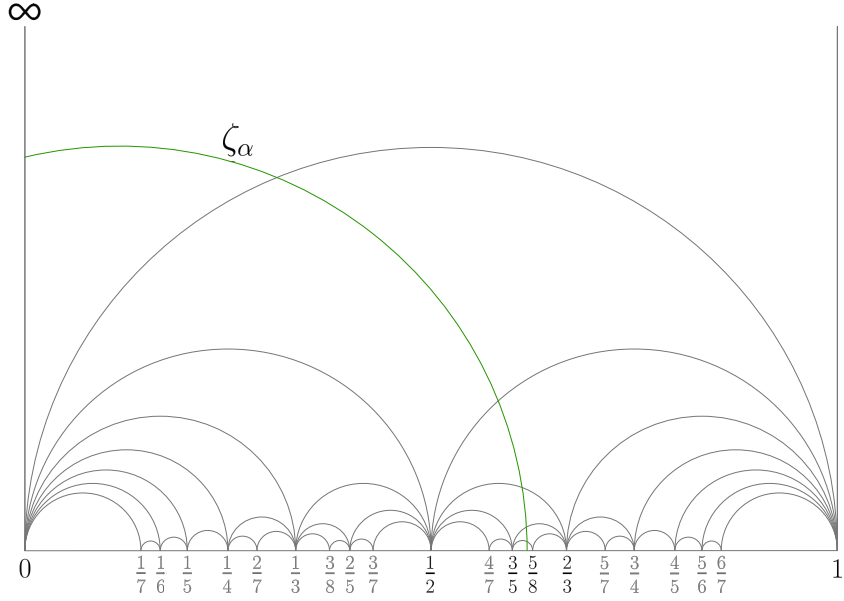} 
  \caption{An image of a geodesic ray $\zeta_{\alpha}$ intersecting the Farey complex $\mathcal{F}$ with (some of the) convergents shown in bold. The endpoint of $\zeta_\alpha$ is $\alpha=\frac{\sqrt{5}-1}{2}$. The convergents are $\infty,0,1,\frac{1}{2},\frac{2}{3},\frac{3}{5},\frac{5}{8},\ldots$.}
  \label{Conv}
\end{figure}

Let $\zeta_\alpha$ be a geodesic ray starting at $I$ and terminating at some point $\alpha>0$. For each fan in the cutting sequence of $\zeta_\alpha$ with $\mathcal{F}$, there is a common vertex of all the triangles in this fan. Using Theorem \ref{thma} and taking truncations of $\overline{\alpha}$, it is easy to show that these vertices are exactly the convergents of $\overline\alpha$. Thus, when a cutting sequence of $\zeta_\alpha$ with $\mathcal{F}$ changes fan, the vertices of the edge at which it changes fan are both convergents of $\overline{\alpha}$. Because $\zeta_\alpha$ passes through pairs of  edges of a triangle, we can similarly recover the set of convergents by taking all the vertices which belong to two or more edges with which the geodesic ray intersects, as well as taking the point at $\infty$. In other words, if two edges of the cutting sequence have a common vertex in $\mathcal{F}$, this vertex is a convergent of  $\overline{\alpha}$. Similarly, if we know a vertex is a convergent, then it is either the point at $\infty$ or the endpoint of at least two edges of $\mathcal{F}$ in our cutting sequence (for $\alpha>1$, $\infty$ will also be the endpoint of at least two edges). See Fig.~\ref{Conv}.

\subsection{Constructing the Dicrete Mulitplicative Map $\overline{n}$}\label{TFC}

Let $n^*:=\begin{psmallmatrix} \sqrt{n} & 0\\ 0 & \frac{1}{\sqrt{n}}\end{psmallmatrix} \in{PSL_2(\mathbb{R})}$ and define $\frac{1}{n^*}:=(n^*)^{-1}$ for $n\in\mathbb{N}$. These two maps scale both $\mathbb{H}$ and $\mathcal{F}$ by a factor of $n$ and $\frac{1}{n}$, respectively. In particular, they multiply the real axis by $n$ and $\frac{1}{n}$. These maps do not preserve $\mathcal{F}$ and we will refer to the images of $\mathcal{F}$ under these maps as $n\mathcal{F}$ and $\frac{1}{n}\mathcal{F}$ respectively. It is worth noting that both of these maps preserve the line $I$ between $0$ and $\infty$, which is our conventional starting edge for our geodesic rays in $\mathcal{F}$. The initial direction of departure is also preserved since $n^*$ and $\frac{1}{n^*}$ preserve the orientation of the $\mathbb{H}$. It follows that for all geodesic rays $\zeta$ starting at $I$, $(\zeta,\frac{1}{n}\mathcal{F})=(n^*(\zeta),\mathcal{F})$. As a result, we can view the map multiplying continued fractions by integer $\overline{n}:\overline{\alpha}\rightarrow{\overline{n\alpha}}$ in terms of a map between $\mathcal{F}$ and $\frac{1}{n}\mathcal{F}$ which preserves $\zeta_\alpha$, a geodesic ray starting at $I$ and terminating at $\alpha$. Explictly, we can express $\overline{n}$ as the map between the pairs $\overline{n}:(\zeta_\alpha,\mathcal{F})\rightarrow{(\zeta_\alpha,\frac{1}{n}\mathcal{F})}$.
%For any group $G$ acting in $\mathbb{H}$ we can construct a non-unique fundamental domain $D$. We can construct a structure $T$, consisting of edges and vertices, which lie in $D$ (including $\partial{D}$) and look at the image of $T$ under the action of all elements of $G$. We denote this set of images as the pair $(G,T)$. We view this structure as being embedded on our initial tile $D$ and since $G$ covers $\mathbb{H}$ with $D$ in a regular way, $(G,T)$ has a regular structure induced by this embedding. Note that this structure need not be connected and may include partial geodesics. For the pair $(G,D)$, we refer to the image of the boundary of our tile $D$ under the action of all elements of $G$.

Since the $\overline{n}$ map in this context is dependent upon the $\frac{1}{n}$ map, we have only described the $\overline{n}$ map via continuous action on $\mathbb{H}$. Instead we want to describe $\overline{n}$ on discrete structures. To find such a discrete map we will claim that for any natural number $n$, there exists a polygon $P_n$ with side pairings and two decorated copies of $P_n$, $T_{\{1,n\}}$ and $T_{\{n,n\}}$, such that $T_{\{1,n\}}$ tessellates $\mathcal{F}$ and $T_{\{n,n\}}$ tessellates $\frac{1}{n}\mathcal{F}$, under the group action induced by the side pairings of $P_n$.  We will take $P_n$ containing the $y$-axis $I$ and will take this edge to be our starting edge, unless otherwise stated. Then, we  express our geodesic ray $\zeta_\alpha$ as a collection of ordered sub-paths $\bigcup^{\infty}_{i=1}{\zeta_{i,\alpha}}$ intersecting the tessellation induced by $P_n$, such that each sub-path $\zeta_{i,\alpha}$ is entirely contained in some image of $P_n$ in this tessellation. Then, the ordered product of the cutting sequences derived by the sub-paths $\bigcup^{\infty}_{i=1}{\zeta_{i,\alpha}}$ is equivalent the cutting sequences of $\zeta_\alpha$. 
%That is, $\bigoplus^{\infty}_{i=1}({\zeta_{i,\alpha}},T_{\{1,n\}})=(\zeta_{1,\alpha},T_{\{1,n\}})\cdot(\zeta_{2,\alpha},T_{\{1,n\}})\cdot\ldots=(\zeta_\alpha,\mathcal{F})$ and $\bigoplus^{\infty}_{i=1}({\zeta_{i,\alpha}},T_{\{n,n\}})=(\zeta_{1,\alpha},T_{\{n,n\}})\cdot(\zeta_{2,\alpha},T_{\{n,n\}})\cdot\ldots=(\zeta_\alpha,\frac{1}{n}\mathcal{F})$.
 In particular, replacing $T_{\{1,n\}}$ with $T_{\{n,n\}}$ encodes the multiplication discrete map $\overline{n}:\overline{\alpha}\rightarrow{\overline{n\alpha}}$.

\subsubsection{Common Structure of $\mathcal{F}$ and $\frac{1}{n}\mathcal{F}$}\label{struc}

For a tiling $T$, any element in $Isom^+(T)$ preserves both $T$ and the orientation, so any subgroup of $Isom^+(T)$ also preserves $T$. We can say more, any subgroup $G<Isom^+(T)$ produces a fundamental domain $F$, which when imbued with the correct structure, tessellates $T$. This extra structure is  $F_T:=T\cap{F}$ embedded into $F$ as a \textit{decoration}. We refer to $F\cup{F_T}$ as a \textit{decorated tile} of $G$. 

Since $Isom^+(\mathcal{F})=SL_2(\mathbb{Z})$ is the maximal orientation-preserving group which preserves $\mathcal{F}$, one can show that $Isom^+(\frac{1}{n}\mathcal{F})=\{ n^{-1}\circ{A}\circ{n} : A\in{SL_2(\mathbb{Z})}\}$ is the maximal orientation preserving group which preserves $\frac{1}{n}\mathcal{F}$. We can view an element of this form as a composition of maps: first the map scaling $\frac{1}{n}\mathcal{F}$ to $\mathcal{F}$, followed by an isomorphism of $\mathcal{F}$ and finally the map scaling $\mathcal{F}$ back to $\frac{1}{n}\mathcal{F}$. We can also write $Isom^+(\frac{1}{n}\mathcal{F})=\bigg\{\begin{psmallmatrix} a & \frac{b}{n}\\ nc & d \end{psmallmatrix} \in SL_2(\mathbb{R}): \begin{psmallmatrix} a & b\\ c & d \end{psmallmatrix}\in{SL_2(\mathbb{Z})} \bigg\}$ and $Isom^+(\frac{1}{n}\mathcal{F})$ takes on a natural group structure induced by $Isom^+(\mathcal{F})$.

We can recover a common subgroup of the group of isomorphisms for $\mathcal{F}$ and $\frac{1}{n}\mathcal{F}$ by taking the intersection of $Isom^+(\mathcal{F})$ and $Isom^+(\frac{1}{n}\mathcal{F})$. $Isom^+(\mathcal{F})\cap{Isom^+(\frac{1}{n}\mathcal{F})}=\Gamma_0(n):=\big\{\begin{psmallmatrix} a & b\\ c & d \end{psmallmatrix} \in SL_2(\mathbb{Z}): c\equiv{0} \,\text{(mod n)} \big\}$. $\Gamma_0(n)$ is a subgroup of both $Isom^+(\mathcal{F})$ and $Isom^+(\frac{1}{n}\mathcal{F})$ by construction and therefore preserves the structure of both $\mathcal{F}$ and $\frac{1}{n}\mathcal{F}$. As a result, any fundamental domain $D$ of $\Gamma_0(n)$ with decoration $D\cap{\mathcal{F}}$ or $D\cap{\frac{1}{n}\mathcal{F}}$ will tessellate $\mathcal{F}$ or $\frac{1}{n}\mathcal{F}$, respectively. 

Finding a fundamental domain $D$ for $\Gamma_0(n)$ has been relatively well studied by mathematicians in \cite{AGM} and \cite{AFD}. We can recover the $\overline{n}$ map by replacing the decoration $D\cap{\mathcal{F}}$ of $D$ with the decoration $D\cap{\frac{1}{n}\mathcal{F}}$. Since $\Gamma_0(n_1)<\Gamma_0(n_2)$ if and only if $n_2\mid{n_1}$, given a fundamental domain of $\Gamma_0(n)$, we can also embed structure to tessellate $\frac{1}{d}\mathcal{F}$  for all $d\mid{n}$ and as a result can recover the $\overline{d}$ map. 

%For such fundamental domains, we can note that there is a symmetry in the line $\frac{1}{2}$ to $\infty$. In particular it follows that $\Gamma_0(n)$ has a fundamental domain which respects this symmetry. 

It will often be useful to tell when two vertices are neighbours in both $\mathcal{F}$ and $\frac{1}{n}\mathcal{F}$. The following lemma provides an if and only if condition.

\begin{lem}\label{lem2} Two points $A$ and $B$ are neighbours in both $\mathcal{F}$ and $\frac{1}{n}\mathcal{F}$ if and only if they have reduced form $\frac{a}{cn_{1}}$ and $\frac{b}{dn_2}$ with $n=n_1 n_2$ and $|{adn_2-bcn_1}|=1$.
\end{lem}

\begin{proof}
$(\Rightarrow)$: If $A$ and $B$ are neighbours in $\mathcal{F}$ and $\frac{1}{n}\mathcal{F}$, then it follows that $n\cdot{A}$ and $n\cdot{B}$ are also neighbours in $\mathcal{F}$. Let $A=\frac{a'}{c'}$ and $B=\frac{b'}{d'}$ be in reduced form, then since $A$ and $B$ are neighbours in $\mathcal{F}$, $|a'd'-b'c'|=1$. Let $g_1:=gcd(c',n)$ and $h_1:=gcd(d',n)$,  then we can write $c'=cg_1$ and $d'=dh_1$ for some $c,d\in\mathbb{N}$ and take $g_2=\frac{n}{g_1}$ and $h_2=\frac{n}{h_1}$. We compute $gcd(c,g_2)=gcd(c,a')=1$ and $gcd(d,h_2)=gcd(d,b')=1$, and therefore, the points $n\cdot{A}=\frac{g_2a'}{c}$ and $n\cdot{B}=\frac{h_2b'}{d}$ are in reduced form. Since $n\cdot{A}$ and $n\cdot{B}$ are neighbours in $\mathcal{F}$, it follows that $|a'g_2d-b'h_2c|=1$. We require that $gcd(g_2,h_2)=1$, since otherwise $|a'g_2d-b'h_2c|\equiv{0}$ mod $gcd(g_2,h_2)$, which for $gcd(g_2,h_2)\neq{1}$ would lead to a contradiction to $A$ and $B$ being neighbours in $\mathcal{F}$. It follows from the fact that $A$ and $B$ are neighbours in $\mathcal{F}$ and from writing $c'=cg_1$ and $d'=dh_1$, that $gcd(g_1,h_1)=1$. We observe that since $gcd(g_1,h_1)=1$, $g_1=gcd(g_1,n)=gcd(g_1,h_1h_2)=gcd(g_1,h_2)$. Similarly, we observe that $h_2=gcd(h_2,n)=gcd(h_2,g_1g_2)=gcd(h_2,g_1)$ and so $g_1=h_2$. By a similar procedure we find that $g_2=h_1$ and the result follows by relabelling $a=a'$, $c=c'$, $n_1=g_1$ and $n_2=h_1$.

$(\Leftarrow):$ Let $A=\frac{a}{cn_{1}}$ and $B=\frac{b}{dn_2}$, with $n=n_1 n_2$ and $|{adn_2-bcn_1}|=1$. Since $|{adn_2-bcn_1}|=1$ we see that $A$ and $B$ are neighbours in $\mathcal{F}$. Also $n\cdot{A}=\frac{an_2}{c}$ and $n\cdot{B}=\frac{bn_1}{d}$ in reduced form and $|{an_2d-bn_1c}|=|{adn_2-bcn_1}|=1$. Therefore, $n\cdot{A}$ and $n\cdot{B}$ are neighbours in $\mathcal{F}$. By rescaling we now see that $A$ and $B$ are neighbours in $\frac{1}{n}\mathcal{F}$ as required.
\end{proof}

The condition that $A$ and $B$ have reduced form $\frac{a}{cn_{1}}$ and $\frac{b}{dn_2}$ with $n=n_1 n_2$ and $|{adn_2-bcn_1}|=1$, translates to saying that if $A$ and $B$ are neighbours of this form in either $\mathcal{F}$ or $\frac{1}{n}\mathcal{F}$, then necessarily they are neighbours in both $\mathcal{F}$ and $\frac{1}{n}\mathcal{F}$. A direct result of this, is that two points of the form $\frac{a}{nc}$ and $\frac{b}{d}$ are neighbours in $\mathcal{F}$ if and only if they are neighbours in $\frac{1}{n}\mathcal{F}$. This implies that for any two neighbours $\frac{a}{nc}$ and $\frac{b}{d}$ in both $\mathcal{F}$ and $\frac{1}{n}\mathcal{F}$, there is an element of the form $\begin{psmallmatrix} a & \pm{b}\\ nc & \pm{d} \end{psmallmatrix}\in\Gamma_0(n)\subset{SL_2}(\mathbb{Z})$ which maps $\infty$ to $\frac{a}{nc}$ and $0$ to $\frac{b}{d}$. The reverse is also true, any element of $\Gamma_0(n)$ maps the vertices $0$ and $\infty$, to a pair of neighbours in both $\mathcal{F}$ and $\frac{1}{n}\mathcal{F}$.

For any point of the form $A=\frac{a}{nc}$ and any two consecutive neighbours $B_1=\frac{b_1}{d_1}$ and $B_2=\frac{b_2}{d_2}$ of $A$ in $\mathcal{F}\cap\frac{1}{n}\mathcal{F}$, we can always find a map  $\phi\in\Gamma_0(n)$, such that $\phi(0)=A$, $\phi(\infty)=B_1$ and $\phi(\frac{1}{n})=B_2$ (up to relabelling). Since $\Gamma_0(n)$ preserves both $\mathcal{F}$ and $\frac{1}{n}\mathcal{F}$, the number of neighbours that $A$ has between $B_1$ and $B_2$ in $\mathcal{F}$ (or equivalently in $\frac{1}{n}\mathcal{F}$) will be equivalent to the number of neighbours that $0$ has between $\infty$ and $\frac{1}{n}$ in $\mathcal{F}$ (or in $\frac{1}{n}\mathcal{F}$). Similarly, for any point of the form $B=\frac{b}{d}$ ($gcd(n,d)=1$) and any two consecutive neighbours $A_1=\frac{a_1}{nc_1}$ and $A_2=\frac{a_2}{nc_2}$ of $B$ in $\mathcal{F}\cap\frac{1}{n}\mathcal{F}$, the number of neighbours that $B$ has between $A_1$ and $A_2$ in $\mathcal{F}$ (or in $\frac{1}{n}\mathcal{F}$) will be equivalent to the number of neighbours that $\infty$ has between $0$ and $1$ in $\mathcal{F}$ (or in $\frac{1}{n}\mathcal{F}$). We summarise this in the following table:

\begin{center}
\begin{tabular}{ | M{3cm} | M{4.5cm}| M{4.5cm} | } 
 \hline
 Points of the form & Number of neighbours in $\mathcal{F}$ between consecutive neighbours in $\mathcal{F}\cap{\frac{1}{n}}\mathcal{F}$ & Number of neighbours in $\frac{1}{n}\mathcal{F}$ between consecutive neighbours in $\mathcal{F}\cap{\frac{1}{n}}\mathcal{F}$ \\ 
 \hline
 & & \\
$\frac{a}{nc}$ & $0$ & $n-1$ \\ 
& & \\
$\,\,\,\,$ $\frac{b}{d}$, $\,\,$ \footnotesize{$gcd(n,d)=1$} & $n-1$ & $0$ \\ 
& & \\
  \hline

\end{tabular}
\end{center}

This information is used to prove the following result.

\begin{pro}\label{pro2} If a continued fraction $\overline\alpha$ has a convergent denominator $q_k$, such that $n\mid{q_k}$ for $n\in\mathbb{N}$ and $n<q_k$, then $B(n\alpha)\geq{n}$. Further, if $\frac{p_k}{q_k} =\frac{p_k}{nq'_k}$ is a convergent of $\overline\alpha$, $\frac{p_k}{q'_k}$ is a convergent of $\overline{n\alpha}$.
\end{pro}

\begin{proof} Let $A=\frac{p_k}{q_k}$ be a convergent of $\overline\alpha$ with geodesic representative $\zeta_\alpha$ in $\mathbb{H}$, such that $n\mid{q_k}$ and $n\in\mathbb{N}$. Then $A$ is a common vertex of a fan in the cutting sequence of $\zeta_\alpha$ with $\mathcal{F}$. and so, at least two edges of the cutting sequence have $A$ as an endpoint. Let $B=\frac{r}{s}$ and $C=\frac{t}{u}$ be the other two endpoints of two such edges, with $\bigtriangleup{ABC}\in\mathcal{F}$. Since $A$ is a neighbour of both $B$ and $C$ in $\mathcal{F}$, $gcd(q_k,s)=gcd(n,s)=gcd(q_k,u)=gcd(n,u)=1$. From Lemma \ref{lem2}, the edges $AB$ and $AC$ are in $\mathcal{F}\cap\frac{1}{n}\mathcal{F}$. By the above paragraph, there exists a map in $\Gamma_0(n)$ which takes $\infty$ to $A$, $0$ to $B$ and $1$ to $C$ (up to relabelling $B$ and $C$). Since $AB$ and $AC$ are both in $\mathcal{F}\cap\frac{1}{n}\mathcal{F}$, there are $n-1$ edges between $AB$ and $AC$ in $\frac{1}{n}\mathcal{F}$, all of which $\zeta_\alpha$ passes through. All these edges share $A$ as an endpoint and so, they are all edges in the same fan. It follows from this, that the fan $\zeta_\alpha$ forms with $\frac{1}{n}\mathcal{F}$ containing both $AB$ and $AC$, contains at least $n$ triangles. Therefore, the cutting sequence $(\zeta_\alpha,\frac{1}{n}\mathcal{F})$ contains a partial quotient with value at least $n$. 

Since $n<q_k$, there exists a $q'_k>1$ such that $q_k=nq'_k$. By the above argument $A=\frac{p_k}{q_k}$ is a common vertex of a fan in the cutting sequence $(\zeta_\alpha,\frac{1}{n}\mathcal{F})$. When we rescale using the $n$-scaling map, $\frac{p_k}{q_k}$ in $\frac{1}{n}\mathcal{F}$ maps to $\frac{p_k}{q'_k}$ in $\mathcal{F}$, which is a common vertex in the cutting sequence $(n^*(\zeta_\alpha),\mathcal{F})$. Therefore, $\frac{p_k}{q'_k}$ is a convergent of $\overline{n\alpha}$. Since $q'_k>1$, it follows that $\frac{p_k}{q'_k}$ is not the common vertex of the first fan but necessarily of some fan after. As a result, the partial quotient of $\overline{n\alpha}$ with value at least $n$ is not the first partial quotient.  By definition $B(n\alpha)\geq{a^{(n)}_i}$ for all $i\in\mathbb{N}$ and since there exists an $a^{(n)}_i\geq{n}$, it follows that $B(n\alpha)\geq{n}$.
\end{proof}

We can improve on this result by taking all such triangles in the fan with the common vertex $\frac{p_k}{q_k}$ and subdividing each of these $n$ times, when taking $\frac{1}{n}\mathcal{F}$. There are $a_k$ such triangles, where $a_k$ is the $k$-th partial quotient. Note that there may be extra terms in this fan (added either side). As a result, we get the following corollary.

\begin{cor} If a continued fraction $\overline\alpha$ has a convergent denominator $q_k$, such that $n\mid{q_k}$ for $n\in\mathbb{N}$ and $n<q_k$, then $B(n\alpha)\geq{n{a_k}}$.
\end{cor}

\subsubsection{Fundamental domains of $\Gamma_0(n)$}\label{RSK}
All results and constructions in this sub-section are contained in \cite{AGM}. 

Fundamental domains of $\Gamma_0(n)$ have been well studied with relation to modular forms. Notably, R.S. Kulkarni gives an explicit construction of a fundamental domain (with side pairings) using \textit{Farey symbols} in \cite{AGM}. This is the construction which we will use and as such we recall important results for ease. 
%Here we will restrict to the case that $n=p$ for $p$ prime, but the result holds for general integers.

A \textit{Farey Sequence} is a sequence of vertices in $\mathcal{F}$, $\{\infty,x_0,\ldots,x_r,\infty\}$, such that each consecutive pair of vertices $x_i$ and $x_{i+1}$ are neighbours in $\mathcal{F}$ and there is some $i\in\{0,\ldots,r\}$ with $x_i=0$. 
Given a Farey sequence, we construct a \textit{Farey symbol} $\sigma$ by identifying each pair of consecutive vertices $x_i$, $x_{i+1}$ with one of the following intervals:

\begin{enumerate}
%\item We construct a geodesic edge between $x_i$ and $x_{i+1}$ and another edge between $x_j$ and $x_{j+1}$, for some $i,j\in\{0,\ldots,n\}$. These edges are identified by the map taking the vertex $x_i$ to $x_{j+1}$ and $x_{i+1}$ to $x_j$. We refer to such pairs of edges as \textit{free sides} of $P$. We represent this information in the following way:
\item A \textit{free interval} with label $a$ such that there is another pair of consecutive vertices $x_j$, $x_{j+1}$, which form a free interval and have the same label,
\end{enumerate}
\begin{center}
$\fs{x_i}{a}{x_{i+1}}$ $\quad$ and \, $\fs{\hphantom{_{+1}}x_j}{a}{x_{j+1}}$
\end{center}

\begin{enumerate}[resume]
\item An \textit{even interval},
\end{enumerate}
\begin{center}
$\,\,\fs{x_i}{\circ}{x_{i+1}}$
\end{center}

\begin{enumerate}[resume]
\item An \textit{odd interval}.
\end{enumerate}
\begin{center}
$\,\,\fs{x_i}{\bullet}{x_{i+1}}$
\end{center}

An example of a Farey symbol is:
 $$\bigg\{{\sfssss{\infty }{1}{\frac{0}{1}}{\bullet}{\frac{1}{2}}{\bullet}{\frac{1}{1}}{1}{\infty}} \bigg\} $$
 \begin{remark} This Farey symbol corresponds to a fundamental domain for $\Gamma_0(7)$. 
 \end{remark}

For each Farey symbol $\sigma$, we can then construct a \textit{special polygon} $P_\sigma$ with edge identifications induced by the interval type, as seen below. We will see shortly that every special polygon is a fundamental domain for some finite index subgroup of $PSL_2(\mathbb{Z})$ and every finite index subgroup of $PSL_2(\mathbb{Z})$ admits a special polygon as a fundamental domain. We construct $P_\sigma$ as follows:

\begin{enumerate}
\item Given two free intervals $\sfs{x_i}{a}{x_{i+1}}$  and  $\sfs{x_j}{a}{x_{j+1}}$,  we construct a geodesic edge between $x_i$ and $x_{i+1}$ and another between $x_j$ and $x_{j+1}$. These edges are identified by the map taking the vertex $x_i$ to $x_{j+1}$ and $x_{i+1}$ to $x_j$. We refer to such pairs of edges as \textit{free sides} of $P_\sigma$. 

\item Given an even interval $\sfs{x_i}{\circ}{x_{i+1}}$, we construct a geodesic edge between $x_i$ and $x_{i+1}$. The edge maps to itself by mapping $x_i$ to $x_{i+1}$ and vice versa (by the elliptic involution about the midpoint of the edge between $x_i$ and $x_{i+1}$). We refer to such an edge as an \textit{even side} of $P_\sigma$. 

\item Given an odd interval $\sfs{x_i}{\bullet}{ x_{i+1}}$, we take the unique triangle between $x_i$, $x_{i+1}$ and $x_i\oplus{x_{i+1}}$. We take $y_i$ to be the centre of this triangle and construct the geodesic edges between $x_i$ and $y_i$ and between $y_i$ and $x_{i+1}$. See Fig.~\ref{p7} and \ref{p72}. These edges are identified by the map which preserves $y_i$ and maps $x_i$ to $x_{i+1}$, $x_{i+1}$ to $x_i\oplus{x_{i+1}}$ and $x_i\oplus{x_{i+1}}$ to $x_i$. We refer to such  edges as an \textit{odd sides} of $P_\sigma$. 
\end{enumerate}

%
%We refer to $\hspace{-20pt}$ $\sfs{\hphantom{_{+1}}x_i}{a}{x_{i+1}}$  $\hspace{-20pt}$ $\sfs{\hphantom{_{+1}}x_i}{\circ}{x_{i+1}}$  $\hspace{-20pt}$ $\sfs{\hphantom{_{+1}}x_i}{\bullet}{x_{i+1}}$ as a \textit{free interval}, \textit{even interval} and \textit{odd interval} respectively.

% We refer to a Farey sequence together with a choice of interval for each consecutive pair of vertices as a \textit{Farey symbol}. Every Farey symbol is in one-to-one correspondence with a Farey symbol and we will often identify these objects.
% 

To construct our side pairings we note that if we have a pair of vertices in $\mathcal{F}$, $x_i=\frac{a_i}{b_i}$ and $x_{i+1}=\frac{a_{i+1}}{b_{i+1}}$ and we want a map in $SL_2(\mathbb{Z})$ taking $x_i$ to $x_{j+1}=\frac{a_{j+1}}{b_{j+1}}$ and $x_{i+1}$ to $x_j=\frac{a_j}{b_j}$, then this map will be of the form:

 \begin{equation}\label{eq1}
 \phi:=\begin{pmatrix} a_{j}b_{i}+a_{j+1}b_{i+1} & -a_ia_j-a_{i+1}a_{j+1} \\ b_ib_j+b_{i+1}b_{j+1} & -a_{i}b_{j}-a_{i+1}b_{j+1} \end{pmatrix}
 \end{equation}

For any Farey symbol $\sigma$, there is a collection of maps (defined above) corresponding to these edge identifications. We define $\Phi_\sigma$ to be group generated by the edge identifications of $\sigma$. By the Poincar\'e Polyhedron Theroem, $P_\sigma$ is a fundamental domain for $\Phi_\sigma$. Theorem (6.1) in \cite{AGM} states that the edge identifications for any Farey symbol $\sigma$, form an independent set of generators for $\Phi_\sigma$. Moreover, the following Theorem explains the importance of special polygons.

\begin{theorem}[Theorem (3.2) and (3.3), R.S. Kulkarni, \cite{AGM}] Every special polygon is a fundamental domain for a finite index subgroup of $PSL_2(\mathbb{Z})$, which is generated by the side pairings. Every finite index subgroup of $PSL_2(\mathbb{Z})$ admits a special polygon as a fundamental domain.
\end{theorem}
Assuming we have a Farey symbol $\sigma$ with $\Phi_\sigma=\Gamma_0(n)$ for some $n$ an positive integer, every edge identification $\phi$ (as in (\ref{eq1})) must satisfy $b_ib_j+b_{i+1}b_{j+1}\equiv{0}$ mod $n$. Therefore, for two pairs of neighbours in $\mathcal{F}$, $x_i$ and $x_{i+1}$, and $x_j$ and $x_{j+1}$, there is a transformation in $\Gamma_0(n)$ which maps the edge between $x_i$ and $x_{i+1}$, and $x_j$ and $x_{j+1}$ if and only if $b_i b_{j}+b_{j+1}b_{i+1}\equiv 0$ mod $n$, where $b_k$ is the denominator of $x_k$ in reduced form. In the case that $b_i\neq{b_{j+1}}$ and $b_{i+1}\neq{b_{j}}$, these edges form a pair of free sides. For an odd edge $b_j=b_i+b_{i+1}$ and $b_{j+1}=b_i$ and for an even edge $b_j=b_i$ and $b_{j+1}=b_{i+1}$. Therefore, two neighbours in $\mathcal{F}$, $x_i$ and $x_{i+1}$,  form an odd edge if and only if  $b_i^2+ b_{i}b_{i+1}+b^2_{i+1}\equiv 0$ mod $n$. Similarly, two neighbours in $\mathcal{F}$, $x_i$ and $x_{i+1}$, form an even edge if and only if  $b_i^2+b^2_{i+1}\equiv 0$ mod $n$. 

For any $\Gamma_0(n)$, we can choose a special polygon $P_\sigma$ (as fundamental domain) such that the $y$-axis $I$ and $I+1$ are paired sides of $P_\sigma$. In particular, we can take $x_0=0$ and $x_r=1$ in the corresponding Farey sequence. This is due to the fact that $\begin{psmallmatrix} 1 & 1\\ 0 & 1 \end{psmallmatrix}\in\Gamma_0(n)$, for all $n\in\mathbb{N}_{>1}$.
% Since $\mathcal{F}$ and $\frac{1}{p}\mathcal{F}$ are symmetric in the line $x=\frac{1}{2}$ and $\Gamma_0(p)$ preserves $\mathcal{F}$ and $\frac{1}{p}\mathcal{F}$, it follows that there exists a fundamental domain of $\Gamma_0(p)$, with vertices which are symmetric in the line $x=\frac{1}{2}$ to $\infty$. See Theorem $13.5$ in \cite{AGM} for more details.

For $p$ prime, we can find a Farey symbol $\sigma$ (for which $P_\sigma$ is a fundamental domain of $\Gamma_0(p)$), in which the vertices are symmetric in the line $x=\frac{1}{2}$ to $\infty$. In other words, the underlying Farey sequence will be of the form $\{\infty,0,x_1,x_2,\ldots{x_2'},x_1',1,\infty\}$, where $x'_i=1-x_i$.
The term $\frac{1}{2}$ will be in every Farey symbol of $\Gamma_0(p)$ for $p\geq{5}$.  This is due to the fact that the line $0$ to $\frac{1}{2}$ separates $\mathbb{H}$ into two regions: one containing the vertex $1$ and the other containing all other neighbours of $0$. Therefore, to get a Farey symbol containing the terms $0$ and $1$, the underlying Farey sequence must either only contain the vertices $\infty$, $0$ and $1$ or  the sequence must contain the vertex $\frac{1}{2}$. If we have either an odd or even interval, the interval identifications are symmetric in the line $\frac{1}{2}$ to $\infty$. However, for free intervals we have antisymmetry, i.e. the free interval labelled $a$ will be replaced with the label $a'$ in this symmetry. Due to the symmetry of the vertices in the line $x=\frac{1}{2}$ to $\infty$ and the pseudo-symmetry of the interval identifications, we will shorten the sequence up to the term $\frac{1}{2}$ (for $p\geq{5}$, since for $p=2,3$ we only use the vertices $\infty$, $0$ and $1$). Similarly, due to identification of the lines $x=0 $ to $\infty$ and $x=1 $ to $\infty$, we will not include these terms in our sequence, with this identification being implicit. For example:
$$\big\{\sfssss{0}{1}{x_1}{1'}{x_2}{\bullet}{x_3}{\circ}{\frac{1}{2}}\mid\text{refl.}\big\}=\big\{\sfssss{0}{1}{x_1}{1'}{x_2}{\bullet}{x_3}{\circ}{\frac{1}{2}}\hspace{-4.6mm}\sfssss{\phantom{\frac{1}{2}}}{\circ}{x'_3}{\bullet}{x'_2}{1}{x'_1}{1'}{1}\big\}$$
We can explicitly state how many odd, even and free intervals there will be in each fundamental domain of $\Gamma_0(n)$. This can be derived from the properties of the quotient space $\mfaktor{\Gamma_0(n)}{\mathbb{H}}$. For a prime $p\geq{5}$ with $p\equiv{1}$ mod $3$  there are exactly two odd intervals (either side of $\frac{1}{2}$), otherwise there are no odd intervals. Similarly, if $p\equiv{1}$ mod $4$ there are exactly two even intervals (either side of $\frac{1}{2}$), otherwise there are no even intervals. If $p\equiv{1}$ mod $3$, the Farey symbol has $\frac{p+2}{3}$ terms, otherwise the Farey symbol has $\frac{p+4}{3}$ terms. 

Given $\Phi$ a subgroup of $PSL_2(\mathbb{Z})$, we can use the Riemann-Hurwitz formula to relate some geometric invariants the quotient space $\mfaktor{\Phi}{\mathbb{H}}$, as follows:
$$ d=3e_2+4e_3+12g+6t-12 $$
\noindent
where:
\begin{itemize}
\item $d$ is the index of $[PSL_2(\mathbb{Z}):\Phi]$
\item $e_2$ is the number orbifold points in $\mfaktor{\Phi}{\mathbb{H}}$ with cone angle $\pi$ (or equivalently the number of even intervals in a corresponding special polygon)
\item $e_3$ is the number orbifold points in $\mfaktor{\Phi}{\mathbb{H}}$ with cone angle $\frac{2\pi}{3}$ (or equivalently the number of odd intervals in a corresponding special polygon)
\item $g$ is the genus of $\mfaktor{\Phi}{\mathbb{H}}$ 
\item $t$ is the number of cusps for $\mfaktor{\Phi}{\mathbb{H}}$
\end{itemize} 

\noindent
For $\Phi=\Gamma_0(n)$:
$$d = n {\displaystyle \prod_{q\mid{n}}}\Big(1+\frac{1}{q}\Big),$$
$$ t = {\displaystyle \sum_{a\mid{n}}}\phi\Big(gcd\Big(a,\frac{n}{a}\Big) \Big),$$
where $q$ is a prime number, $a\in\mathbb{N}$ and $\phi$ is the Euler totient function. 

Calculating this information for $\Gamma_0(p)$, we observe that the quotient space $\mfaktor{\Gamma_0(p)}{\mathbb{H}}$ will have $2$ punctures, $e_2$ even intervals, $e_3$ odd intervals and genus $g$. The above relation then reduces to:
$$p+1=3e_2+4e_3+12g $$
\subsubsection{Decorated tiles of $\Gamma_0(n)$}

\begin{defn} We define $P_n$ to be a special polygon, with the $y$-axis $I$ and $I+1$ as paired sides, which is a fundamental domain for $\Gamma_0(n)$. 
\end{defn}

If we take $P_n$ to be a fundamental domain for $\Gamma_0(n)$, we can construct  $T_{\{1,n\}}:=P_n\cup(P_n\cap\mathcal{F})$ and $T_{\{n,n\}}:=P_n\cup(P_n\cap{\frac{1}{n}\mathcal{F})}$ to be two decorated tiles of $\Gamma_0(n)$ such that $\Gamma_0(n)\cdot{T_{\{1,n\}}}=\mathcal{F}$ and $\Gamma_0(n)\cdot{T_{\{n,n\}}}=\frac{1}{n}\mathcal{F}$. See Fig.~\ref{Tn} for images of $T_{\{1,n\}}$ and $T_{\{n,n\}}$ for $n=7,11$.
%These triangulations are well defined due to the fact $\Gamma_0(n)\cdot{T_{\{1,n\}}}=\mathcal{F}$ and $\Gamma_0(n)\cdot{T_{\{n,n\}}}=\frac{1}{n}\mathcal{F}$. 
We can similarly define $T_{\{d,n\}}:=P_n\cup(P_n\cap\{\frac{1}{d}\mathcal{F}))$ for every $d\mid{n}$ and the decorated tile $T_{\{d,n\}}$ together with the side pairings induced by $\Gamma_0(n)$ encodes sufficient data to recover $\frac{1}{d}\mathcal{F}$, for all $d\mid{n}$.

%If $\zeta'_\alpha$ is a projection of a geodesic ray $\zeta_\alpha$ in $\mathbb{H}$ to $P_n$, we have $(\zeta'_\alpha,T_{\{1,n\}})=(\zeta_\alpha,\mathcal{F})$. Similarly, we have $(\zeta'_\alpha,T_{\{n,n\}})=(\zeta_\alpha,\frac{1}{n}\mathcal{F})$.
For any geodesic ray $\zeta_\alpha$, we can decompose $\zeta_\alpha$ into an ordered collection of sub-paths $\bigcup^{\infty}_{i=1}\zeta_{i,\alpha}$, such that each $\zeta_{i,\alpha}$ is entirely contained in an image of ${P_n}$ under its tessellation by $\Gamma_0(n)$. We will abuse notation and think of each $\zeta_{i,\alpha}$ as a sub-path in $P_n$. Then the cutting sequence of $\zeta_\alpha$ with $\frac{1}{d}\mathcal{F}$ for $d\mid{n}$, is equivalent to ordered product of the cutting sequences for each $\zeta_{i,\alpha}$ with $T_{\{d,n\}}$. 
%Here, we take the cutting sequence of each $\zeta_{i,\alph}$ to have even parity i.e. $(\zeta_{i,\alpha},T_{\{d,n\}})=\{n_0,n_1,\ldots,n_{2m-1}\}$ for $n_0,n_{2m-1}\in\Mathbb{N}\cup\{0\}$ and $n_i\in\mathbb{N}$ otherwise, and define the product of two cutting sequences to be given by appending one to another i.e. . 
Explicitly, $\prod^{\infty}_{i=1}(\zeta_{i,\alpha},T_{\{d,n\}})=(\zeta_{1,\alpha},T_{\{d,n\}})\cdot(\zeta_{2,\alpha},T_{\{d,n\}})\cdot\ldots=(\zeta_\alpha,\mathcal{F})$, where $(\zeta_{1,\alpha},T_{\{d,n\}})\cdot(\zeta_{2,\alpha},T_{\{d,n\}})=\{L^{n_0}R^{n_1}\cdots\}\cdot\{L^{n'_0}R^{n'_1}\cdots\}=\{L^{n_0}R^{n_1}\cdots{L^{n'_0}R^{n'_1}\cdots\}} $. 

\begin{remark} On a technical note, we take $(\zeta_{i,\alpha},T_{\{d,n\}})$ to be the generalised cutting sequence with the canonically induced labelled triangulation on $\mathcal{F}$. This is simply so that we do not need to form a "full" left or right triangle in $T_{\{d,n\}}$, and so the first term in $(\zeta_{i,\alpha}, T_{\{d,n\}})$  and the last term in $(\zeta_{i-1,\alpha}, T_{\{d,n\}})$ are both well-defined and without error.
\end{remark}

%The decorations of $T_{\{1,n\}}$ and $T_{\{n,n\}}$ are finite and as such, our $\overline{n}$ map can be viewed as a discrete map from $T_{\{1,n\}}$ and $T_{\{n,n\}}$. In fact, for any $d\mid{n}$ the map $\overline{d}$ can be viewed as a discrete map from $T_{\{1,n\}}$ and $T_{\{d,n\}}$.

\subsubsection*{Algorithm for Integer Multiplication of a Continued Fraction by $n$.} 
By using the above notions of sub-path, we obtain the following algorithm for multiplying a continued fraction $\overline{\alpha}$ by $n$ some integer:
\begin{enumerate}
\item Construct the fundamental domain $P_n$ of $\Gamma_0(n)$.
\item Construct the decorated tiles of $T_{\{1,n\}}$ and $T_{\{n,n\}}$.
\item Use the continued fraction $\overline{\alpha}$ to algorithmically recover a curve $\zeta_\alpha$ as a sequence of sub-paths $\bigcup\limits^{\infty}_{i=1}\zeta_{i,\alpha}$ intersecting $T_{\{1,n\}}$.
\item Take the cutting sequence of each $\zeta_{i,\alpha}$ with respect to $T_{\{n,n\}}$.
\item Compute $\prod\limits^{\infty}_{i=1}(\zeta_{i,\alpha},T_{\{n,n\}})=(\zeta_{1,\alpha},T_{\{n,n\}})\cdot(\zeta_{2,\alpha},T_{\{n,n\}})\cdot\ldots=(\zeta_\alpha,\mathcal{F})$.
\end{enumerate}

\begin{figure}[hbt]
        \centering
        \begin{subfigure}{0.4\textwidth}
            \centering
            \includegraphics[width=\textwidth]{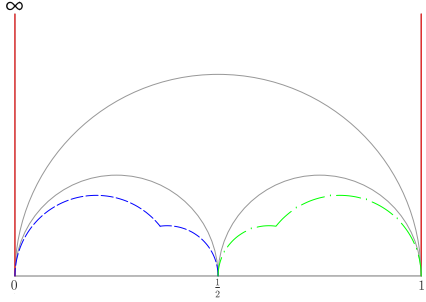}
            \caption{{An image of $T_{\{1,7\}}$.}}    
            \label{p7}
        \end{subfigure}
        \quad\quad
        \begin{subfigure}{0.4\textwidth}  
            \centering 
            \includegraphics[width=\textwidth]{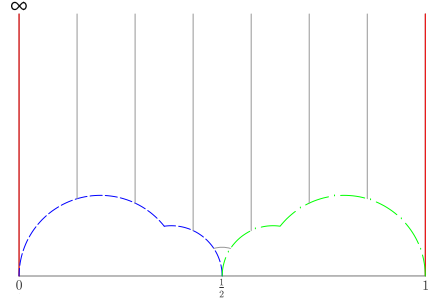}
            \caption{{An image of $T_{\{7,7\}}$.}}    
            \label{p72}
        \end{subfigure}
        \vskip\baselineskip
        \begin{subfigure}{0.4\textwidth}   
            \centering 
            \includegraphics[width=\textwidth]{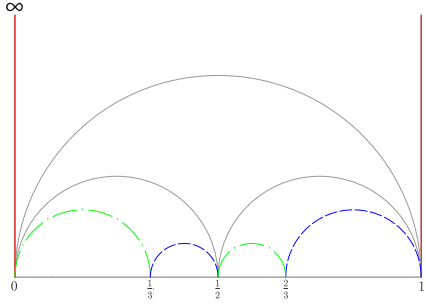}
            \caption{{An image of $T_{\{1,11\}}$.}}    
            \label{p11}
        \end{subfigure}
        \quad\quad
        \begin{subfigure}{0.4\textwidth}   
            \centering 
            \includegraphics[width=\textwidth]{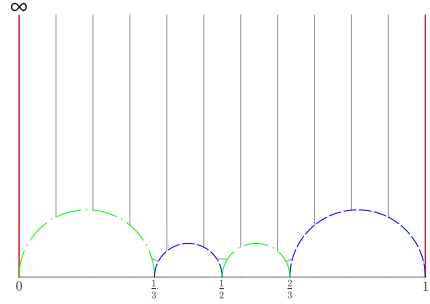}
            \caption{{An image of $T_{\{11,11\}}$.}}    
            \label{p112}
        \end{subfigure}
        \caption{Images of $P_7$ and $P_{11}$, with embedded structure of $\mathcal{F}$  for (a) and (c), and embedded structure of $\frac{1}{7}\mathcal{F}$ and $\frac{1}{11}\mathcal{F}$ on (b) and (d) respectively. Side pairings are indicated by bold, short dashed and long dashed lines. For (a) and (c), the dashed lines are part of $\mathcal{F}$.} 
        \label{Tn}
    \end{figure}
\begin{remark}
The algorithm that we obtain is not particularly useful for explicit computation (such an algortihm can be derived from taking all possible sub-paths), however it does have some useful theoretic properties, some of which we will see later in this paper and some of which we leave for future work. A more explicit algorithm can be found in \cite{MMT}.
\end{remark}
\section{Cutting sequences on $\mfaktor{\Gamma_0(n)}{\mathbb{H}}$}

In this section, we define cutting sequences on triangulated orbifolds and show the bijective relation between positive essentially periodic continued fractions (see definition \ref{esp}.(2.)) and closed curves on these triangulated orbifolds. Using this result, we show that the bounds of essentially periodic continued fractions grow exponentially when iteratively multiplied by a fixed integer. We then relate this result to the convergents of essentially periodic continued fractions and show that every eventually periodic continued fraction multiplies like an essentially periodic continued fraction.
\subsection{Cutting Sequences on $\mfaktor{\Gamma_0(n)}{\mathbb{H}}$}

In Section \ref{CFCS}, we defined cutting sequences of geodesic rays with respect to ideal triangulations of $\mathbb{H}$. The concept of a geodesic ray intersecting a triangle to form a left or right triangle is independent of metric and thus, for any triangulated surface we can define a cutting sequence for a geodesic ray relative to this triangulation.

However, when we take $\mfaktor{\Phi_\sigma}{\mathbb{H}}$, for $\sigma$ some Farey symbol (or equivalently, take the corresponding special polygon $P_\sigma$ and identify sides, which we denote $\sfaktor{P_\sigma}{\sim}$), we are not necessarily left with a surface. Instead, we obtain a two-dimensional \textit{orbifold}. Here we define a \textit{two-dimensional orbifold} to be a surface $S$ (possibly with boundary), with a set of marked points $M$ and a potentially empty set of orbifold points $Q$. In the case that $Q$ is empty, the orbifold will be a surface. When we supply the orbifold with a metric, each element of $M$ will correspond to \textit{cusps} with angle $0$ and each orbifold point will correspond to a \textit{cone point} with angle $\frac{2\pi}{k}$ for some $k\in\mathbb{N}>{1}$. We will only consider orbifolds with empty boundary, at least one cusp (element of $M$) and a potentially empty set of orbifold points $Q$, which can be decomposed into two disjoint subsets: $E_2$, the set of orbifold points with angle $\pi$ and $E_3$, the set of orbifold points with angle $\frac{2\pi}{3}$. When we take the corresponding special polygon $P_\sigma$ and identify sides via the side pairings, we see that elements of $E_2$ correspond to the central points the even edges, elements of $E_3$ correspond to the interior vertices of $P_\sigma$ formed by the odd edges and  elements of $M$ correspond to vertices on the boundary of $\mathbb{H}$ quotient the side pairings. 

\begin{remark} Following the definitions in Section \ref{RSK}, the number of cusps is exactly given by $|M|=t$ and the number of orbifold points with cone angles $\pi$ and $\frac{2\pi}{3}$ are $|E_2|=e_2$ and $|E_3|=e_3$, respectively.
\end{remark}

We define an \textit{arc} $\gamma$ on an orbifold $\mathcal{O}$ to be a geodesic path, which is disjoint from $M\cup{Q}$ except from its endpoints, with the following properties:
\begin{itemize}
\item The endpoints of $\gamma$ are contained in $M\cup{Q}$ and at least one endpoint is in $M$,
\item The only self-intersections of $\gamma$ occur at the endpoints of $\gamma$, if at all,
\item If $\gamma$ bounds a monogon i.e. both endpoints of $\gamma$ are at the same point in $M$, then this monogon either contains one element of $M$, one element of $E_3$ or two elements of $E_2$.
\end{itemize}
%The following definition of a \textit{triangulation} on an \textit{orbifold} is inspired by \cite{CATO}.
%We define \textit{an orbifold} $\mathcal{O}$ to be a surface $S$ with boundary $\partial{S}$, a non-empty set of marked points $P$, a set of orbifold points $Q$. We require that every boundary component of $S$ contains at least one marked point and that $P\cap{Q}=\emptyset$ and $Q\cap\partial{S}=\emptyset$. In our setting, the boundary will be empty and our orbifold points $Q$ will be split into two disjoint sets: $E_2$, the set of orbifold points with angle $\pi$ and $E_3$, the set of orbifold points with angle $\frac{2\pi}{3}$. We define an \textit{arc} $\gamma$ on $\mathcal{O}$ to be a geodesic path in $\mathcal{O}$, disjoint from $P\cup{Q}$ except at the endpoints, satisfying the following properties:
%\begin{itemize}
%\item One of the three statements holds:
%\begin{itemize}
%\item both are elements of $P$;
%\item one endpoint is an element of $P$ and the other is an element of $E_2$.
%\item one endpoint is an element of $P$ and the other is an element of $E_3$, in which case we refer to the arc as an \textit{admissable arc}.
%\end{itemize}
%\item $\gamma$ has no self-intersections except for (possibly) the endpoints.
%\item if $\gamma$ forms a monogon, then this monogon contains either:
%\begin{itemize}
%\item a single element of $P$;
%\item two elements of $E_2$, or;
%\item a single element of $E_3$.
%\end{itemize}
%\end{itemize}
If $\gamma$ has one endpoint in $E_3$, then we will say that $\gamma$ is a \textit{structural arc}. We will say that a pair of arcs $\gamma$, $\gamma'$ are \textit{compatible}, if $\gamma\cap{\gamma'}\subset{M}$ (i.e. $\gamma$ and $\gamma'$ only intersect at endpoints which are also marked points). Similarly, we define a \textit{quotient triangulation} $T$ of an orbifold $\mathcal{O}$ to be a maximal collection of pairwise compatible arcs on $\mathcal{O}$. There are six possible of types triangle that can arise from a quotient triangulation, which we list in Table \ref{de}. 
\begin{table}[!h]
\caption{A table of the six possible types of triangles that can appear in a quotient triangulation and their lifts in $\mathbb{H}$. Elements of $P$ are indicated by $\bullet$, elements of $E_2$ are indicated by $\circ$ and elements of $E_3$ are indicated by $\square$. Dashed lines indicate structural arcs and their lifts.}
\label{de}

  \begin{tabular}{|c|c|c|c|}
      \hline

      Type &   Name    &  Diagram & Lift in $\mathbb{H}$  \\ \hline

(I) & Standard triangle &   
    \belowbaseline[0pt-\heightof{X}]{\begin{tikzpicture}
    \clip (-1.2,-0.2) rectangle (1.2,1.5);
    \filldraw (-0.75,0) circle (2pt);
    \filldraw (0.75,0) circle (2pt);
    \filldraw (0,1.299) circle (2pt);
    \draw (-0.75,0) to (0.75,0);
      \draw (-0.75,0) to (0,1.299);
      \draw (0,1.299) to (0.75,0);
    \end{tikzpicture}  }         
    &  \belowbaseline[0pt-\heightof{X}]{\begin{tikzpicture}
    \clip (-1.2,-0.2) rectangle (1.2,1.5);
    \filldraw (0,0) circle (2pt);
    \filldraw (1,0) circle (2pt);
    \filldraw (-1,0) circle (2pt);
    \draw[domain=0:180] plot ({cos(\x)}, {sin(\x)});
   \draw[domain=0:180] plot ({0.5*cos(\x)+0.5}, {0.5*sin(\x)});
   \draw[domain=0:180] plot ({0.5*cos(\x)-0.5}, {0.5*sin(\x)});
%    \draw (-0.5,0) circle(0.5);
%      \draw (0.5,0) circle(0.5);
%      \draw (0,0) circle(1);
      \draw (-1.2,0) to (1.2,0);
    \end{tikzpicture}    }
  
     \\ \hline
    
    (II) & Self-folded triangle &
   \belowbaseline[0pt-\heightof{X}]{ \begin{tikzpicture}
   \clip (-1.2,-0.2) rectangle (1.2,1.95);
    \filldraw (0,1) circle (2pt);
    \filldraw (0,0) circle (2pt);
    \draw (0,0) to (0,1);
      \draw (0,0) to[in=0, out=45] (0,1.75);
      \draw (0,0) to[in=180, out=135] (0,1.75);
    \end{tikzpicture}  }       
    & \belowbaseline[0pt-\heightof{X}]{\begin{tikzpicture}
    \clip (-1.2,-0.2) rectangle (1.2,1.95);
    \filldraw (0,0) circle (2pt);
    \filldraw (1,0) circle (2pt);
    \filldraw (-1,0) circle (2pt);
      \node[draw=none] at (0.5,0.3) {\small{a}};
        \node[draw=none] at (-0.5,0.3) {\small{a}};
     \draw[domain=0:180] plot ({cos(\x)}, {sin(\x)});
   \draw[domain=0:180] plot ({0.5*cos(\x)+0.5}, {0.5*sin(\x)});
   \draw[domain=0:180] plot ({0.5*cos(\x)-0.5}, {0.5*sin(\x)});
      \draw (-1.2,0) to (1.2,0);
    \end{tikzpicture}  }
      \\ \hline
    
(IIIa) &  Quotient-2 triangle (a) &
   \belowbaseline[0pt-\heightof{X}]{ \begin{tikzpicture}
   \clip (-1.2,-0.2) rectangle (1.2,1.95);
    \filldraw (0,1.75) circle (2pt);
    \filldraw [fill=white] (0,1) circle (2pt);
    \filldraw (0,0) circle (2pt);
    \draw (0,0) to (0,0.925);
      \draw (0,0) to[in=0, out=45] (0,1.75);
      \draw (0,0) to[in=180, out=135] (0,1.75);
    \end{tikzpicture}    }   
    & \belowbaseline[0pt-\heightof{X}]{\begin{tikzpicture}
    \clip (-1.2,-0.2) rectangle (1.2,1.95);
    \filldraw[fill=white] (0,1) circle (2pt);
    \filldraw (1,0) circle (2pt);
    \filldraw (-1,0) circle (2pt);
         \draw[domain=94:180] plot ({cos(\x)}, {sin(\x)});
   \draw[domain=0:86] plot ({cos(\x)}, {sin(\x)});
    \draw (-1,0) to (-1,1.75);
      \draw (1,0) to (1,1.75);
      \draw (-1.2,0) to (1.2,0);
    \end{tikzpicture}  }
       \\ \hline
    
    (IIIb) & Quotient-2 triangle (b) & 
    \belowbaseline[0pt-\heightof{X}]{\begin{tikzpicture}
    \clip (-1.2,-0.2) rectangle (1.2,1.95);
    \filldraw [fill=white] (-0.275,0.85) circle (2pt);
    \filldraw [fill=white] (0.275,0.85) circle (2pt);
    \filldraw  (0,0) circle (2pt);
    \draw (0,0) to (0.25,0.775);
    \draw (0,0) to (-0.25,0.775);
      \draw (0,0) to[in=0, out=30] (0,1.75);
      \draw (0,0) to[in=180, out=150] (0,1.75);
    \end{tikzpicture}        }
    &\belowbaseline[0pt-\heightof{X}]{ \begin{tikzpicture}
    \clip (-1.2,-0.2) rectangle (1.2,1.95);
    \filldraw (0,0) circle (2pt);
    \filldraw (1,0) circle (2pt);
    \filldraw (-1,0) circle (2pt);
    \filldraw[fill=white] (0.5,0.5) circle (2pt);
    \filldraw[fill=white] (-0.5,0.5) circle (2pt);
     \draw[domain=98:180] plot ({0.5*cos(\x)+0.5}, {0.5*sin(\x)});
   \draw[domain=0:82] plot ({0.5*cos(\x)-0.5}, {0.5*sin(\x)});
    \draw[domain=98:180] plot ({0.5*cos(\x)-0.5}, {0.5*sin(\x)});
   \draw[domain=0:82] plot ({0.5*cos(\x)+0.5}, {0.5*sin(\x)});
      \draw[domain=0:180] plot ({cos(\x)}, {sin(\x)});
      \draw (-1.2,0) to (1.2,0);
    \end{tikzpicture} } 
       \\ \hline   
        
(IIIc)* &  Quotient-2 triangle (c) &
    \belowbaseline[0pt-\heightof{X}]{\begin{tikzpicture}
    \clip (-1.2,-0.2) rectangle (1.2,1.95);
    \filldraw [fill=white] (0.25,0.7) circle (2pt);
    \filldraw [fill=white] (0.875,0.9) circle (2pt);
    \filldraw [fill=white] (-0.875,0.9) circle (2pt);
    \filldraw (0,0) circle (2pt);
    \draw (0,0) to[in=270, out=30] (0.25,0.625);
      \draw (0,0) to[in=270, out=0] (0.875,0.8);
      \draw (0,0) to[in=270, out=180] (-0.875,0.825);
      \draw[dotted, domain=6:174] plot ({0.875*cos(\x)},{0.875*sin(\x)+0.875});
    \end{tikzpicture}    }   
    & \belowbaseline[0pt-\heightof{X}]{\begin{tikzpicture}
    \clip (-1.2,-0.2) rectangle (1.2,1.95);
    \filldraw[fill=white] (0,1) circle (2pt);
     \filldraw[fill=white] (-1,1.3) circle (2pt);
      \filldraw[fill=white] (1,1.3) circle (2pt);
    \filldraw (1,0) circle (2pt);
    \filldraw (-1,0) circle (2pt);
         \draw[domain=94:180] plot ({cos(\x)}, {sin(\x)});
   \draw[domain=0:86] plot ({cos(\x)}, {sin(\x)});
    \draw (-1,0) to (-1,1.225);
      \draw (1,0) to (1,1.225);
       \draw (-1,1.375) to (-1,1.75);
      \draw (1,1.375) to (1,1.75);
      \draw (-1.2,0) to (1.2,0);
    \end{tikzpicture}  }
       \\ \hline

(IV) &  Quotient-3 triangle  & 
    \belowbaseline[0pt-\heightof{X}]{\begin{tikzpicture}
    \clip (-1.2,-0.2) rectangle (1.2,1.95);
    \filldraw [fill=white] ([xshift=-2pt,yshift=-2pt]0,1) rectangle ++(4pt,4pt);
    \filldraw (0,0) circle (2pt);
      \draw (0,0) to[in=0, out=45] (0,1.75);
      \draw (0,0) to[in=180, out=135] (0,1.75);
      \draw[dashed] (0,0) to (0,0.93);
    \end{tikzpicture}   } 
     &\belowbaseline[0pt-\heightof{X}]{ \begin{tikzpicture}
    \clip (-1.2,-0.2) rectangle (1.2,1.95);
    \filldraw (1,0) circle (2pt);
    \filldraw (-1,0) circle (2pt);    
    \filldraw[fill=white] ([xshift=-2pt,yshift=-2pt]0,0.7) rectangle ++(4pt,4pt);
       \draw[dashed,domain=77:180] plot ({0.75*cos(\x)-0.25}, {0.75*sin(\x)});
   \draw[dashed,domain=00:103] plot ({0.75*cos(\x)+0.25}, {0.75*sin(\x)});
      \draw[domain=0:180] plot ({cos(\x)}, {sin(\x)});
      \draw (-1.2,0) to (1.2,0);
    \end{tikzpicture}  }
         \\ \hline
         
%(IIIc) &  Quotient-2 triangle (c) &
%    \begin{tikzpicture}
%    \filldraw [fill=white] (0.3,0.8) circle (2pt);
%    \filldraw [fill=white] (0.9,0.9) circle (2pt);
%    \filldraw [fill=white] (-0.9,0.9) circle (2pt);
%    \filldraw (0,0) circle (2pt);
%    \draw (0,0) to[in=270, out=30] (0.3,0.725);
%      \draw (0,0) to[in=270, out=0] (0.9,0.825);
%      \draw (0,0) to[in=270, out=180] (-0.9,0.825);
%      \draw[dotted, domain=4:176] plot ({0.9*cos(\x)},{0.9*sin(\x)+0.9});
%    \end{tikzpicture}       
%    & \begin{tikzpicture}
%    \clip (-1.2,-0.0756) rectangle (1.2,2.1);
%    \filldraw[fill=white] (0,1) circle (2pt);
%     \filldraw[fill=white] (-1,1.5) circle (2pt);
%      \filldraw[fill=white] (1,1.5) circle (2pt);
%    \filldraw (1,0) circle (2pt);
%    \filldraw (-1,0) circle (2pt);
%         \draw[domain=94:180] plot ({cos(\x)}, {sin(\x)});
%   \draw[domain=0:86] plot ({cos(\x)}, {sin(\x)});
%    \draw (-1,0) to (-1,1.425);
%      \draw (1,0) to (1,1.425);
%       \draw (-1,1.575) to (-1,2);
%      \draw (1,1.575) to (1,2);
%      \draw (-1.2,0) to (1.2,0);
%    \end{tikzpicture}  
%       \\ \hline
%    

  \end{tabular}
\end{table}

\begin{remark}* The quotient-2 triangle (IIIc) occurs as a triangulation for exactly one orbifold. This orbifold has three elements in $E_2$ and a single cusp, and the triangle is formed by taking an arc between each point in $E_2$ and the cusp.  Only one subgroup $\Phi$ of $PSL_2(\mathbb{Z})$, given by $\Gamma_3=\langle \begin{psmallmatrix} 0 & -1\\ 1 & 0 \end{psmallmatrix}, \begin{psmallmatrix} 1 & -2\\ 1 & -1 \end{psmallmatrix}, \begin{psmallmatrix} 1 & -1\\ 2 & -1 \end{psmallmatrix} \rangle$, induces a quotient space $\mfaktor{\Phi}{\mathbb{H}}$ that allows such a triangle and only triangles of this type appear on this orbifold. There are two special polygons of $\Gamma_3$ with Farey symbols $\{\sfsss{\infty}{\circ}{0}{\circ}{\pm1}{\circ}{\infty}\}$.
\end{remark}

When we take the Farey complex $\mathcal{F}$ relative to some special polygon $P_\sigma$, it is not hard to see that $\mathcal{F}$ decomposes $P_\sigma$ into triangles of type $(\widetilde{\text{I}})-(\widetilde{\text{IV}})$ minus the lift of any structural arcs (where $(\widetilde{\text{I}})$ is the lift of (I) in $\mathbb{H}$ etc.). As a result, the projection of $\mathcal{F}$ decomposes $\mfaktor{\Phi}{\mathbb{H}}$ into triangles of type {(I)-(IIIc)} or into monogons containing a single element of $E_3$, for $\Phi$ any finite index subgroup of $PSL_2(\mathbb{Z})$. For the monogons containing a single element of $E_3$, we can construct a unique structural arc between this element of $E_3$ and the element of $M$ on the boundary of this monogon. In particular, the projection of $\mathcal{F}$ induces a unique quotient triangulation for all the quotient spaces $\mfaktor{\Phi}{\mathbb{H}}$, for $\Phi$ any finite index subgroup of $PSL_2(\mathbb{Z})$. However, it is not immediately clear when an ideal triangulation $T$ will induce a quotient triangulation of $\mfaktor{\Phi}{\mathbb{H}}$. In particular, it is not obvious that $\frac{1}{d}\mathcal{F}$ will induce a quotient triangulation  of $\mfaktor{\Gamma_0(n)}{\mathbb{H}}$ for $d\mid{n}$. The following lemma gives a sufficient condition for an ideal triangulation $T$ of $\mathbb{H}$ to induce a unique quotient triangulation on  $\mfaktor{\Phi}{\mathbb{H}}$.

%Note that this is a somewhat obvious statement for $T=\mathcal{F}$, since our special polygon $P_\Phi$ will be nicely decomposed into triangles as seen in the "Lift in $\mathbb{H}$" column of Table. \ref{de}. For the monogons (an arc $\gamma$ with endpoint $m$ in $M$) that contain an element $e$ of $E_3$, we can then construct a unique structural arc between $m$ and $e$ to obtain a quotient triangulation $T_\Phi$. However, for arbitrary ideal triangulations preserved by $\Phi$ this is not so obvious. This lemma is required to show that the scaled Farey triangulations $\frac{1}{d}$ do indeed induce a quotient triangulation on $\mfaktor{\Gamma_0(n)}{\mathbb{H}}$ for $d\mid{n}$.

\begin{lem}\label{Triang}
Let $\Phi$ be a finite subgroup of $PSL_2(\mathbb{Z})$ (excluding $\Gamma_3$), let $P_\Phi$ be a special polygon which is a fundamental domain for $\Phi$, and let $T$ be an ideal triangulation of $\mathbb{H}$, which is invariant under $\Phi$. Then the projection of $T$ decomposes $\mfaktor{\Phi}{\mathbb{H}}$ into triangles of type \emph{(I)-(IIIb)} or into monogons containing a single element of $E_3$. In particular, the projection of $T$ induces a unique quotient triangulation $T_\Phi$ of $\mfaktor{\Phi}{\mathbb{H}}$.
\end{lem}

%\noindent
%\textit{Outline of proof}. This proof is split up into multiple parts, which we will number for ease.
%
%
%\begin{enumerate}
% \item Firstly, we will show that if a triangle $\delta$ in the ideal triangulation $T$ of $\mathbb{H}$ that does not contain a lift of an orbifold point $Q$ of $\mfaktor{\Phi}{\mathbb{H}}$, then $\delta$ projects to a triangle of type (I) or (II) in $\mfaktor{\Phi}{\mathbb{H}}$.
% \item We then show that if $\delta$ contains the lift of an orbifold point in $E_2$, then $\delta$ projects to a triangle of type (IIIa)-(IIIb) in $\mfaktor{\Phi}{\mathbb{H}}$.
% \item Finally, we show that if $\delta$ contains the lift of an orbifold point in $E_3$, then $\delta$ projects to a monogon containing a single element of $E_3$. As seen above, we can then construct a unique structural arc between this element of $E_3$ and the element of $M$ on the boundary of this monogon.
%\end{enumerate}

\begin{proof}
\noindent
1. Firstly, we will show that if a triangle $\delta$ in the ideal triangulation $T$ of $\mathbb{H}$ does not contain a lift of an orbifold point  in $\mfaktor{\Phi}{\mathbb{H}}$, then $\delta$ projects to a triangle of type (I) or (II) in $\mfaktor{\Phi}{\mathbb{H}}$.\par
\vspace{\baselineskip}
Let $\delta$ be a triangle in the ideal triangulation $T$ of $\mathbb{H}$, which does not contain a lift of an orbifold point in $\mfaktor{\Phi}{\mathbb{H}}$. Then, the projection of $\delta$ in $\mfaktor{\Phi}{\mathbb{H}}$ will not contain any elements of $Q=E_2\cup{E_3}$. Since $T$ is invariant under $\Phi$, geodesics in $\mathbb{H}$ will project to geodesic arcs in $\mfaktor{\Phi}{\mathbb{H}}$ and these geodesic arcs will be pairwise disjoint except for at $P$. As a result, the projection of $\delta$ will be triangles of type (I) or (II).\par
\vspace{\baselineskip}

\noindent
2. We now show that if $\delta$ contains the lift of an orbifold point in $E_2$, then $\delta$ projects to a triangle of type (IIIa)-(IIIb) in $\mfaktor{\Phi}{\mathbb{H}}$.\par
\vspace{\baselineskip}

\noindent
\textbf{Claim}: If $P_\Phi$ contains an even edge $e$, then any triangulation $T$ preserved by $\Phi$ must contain an edge that runs through $m_e$, where $m_e$ is the fixed point of $\phi_e$, the side pairing induced by the even edge $e$.\par
\vspace{\baselineskip}

\noindent
\textit{Proof of Claim}: First, we assume the opposite, that $m_e$ is not intersected by any edge of $T$. Then, since $e_2$ lies in the interior of $\mathbb{H}$, $m_e$ must lie in the interior of  $\delta$, some triangle in $T$. Two vertices of $\delta$ will lie on one side $e_+$ of the even edge $e$ and one vertex of $\delta$ will lie on the other side $e_-$. Since $\phi_e$ is an elliptic involution of order 2 with fixed point $m_e$, the image of $\phi_e(\delta)$ will contain $m_e$ and have 2 vertices in $e_-$ and one vertex in $e_+$. Since $\phi_e$ is an element of $\Phi$, it follows that $\phi_e(\delta)$ must be a triangle in $T$ (since, $T$ is invariant under $\Phi$). Both triangles $\delta$ and $\phi_e(\delta)$ contain the point $m_e$,  however $\delta\neq{\phi_e(\delta)}$, since the number of endpoints in $e_+$ and $e_-$ are different for $\delta$ and $\phi_e(\delta)$. This implies $\delta$ and $\phi_e(\delta)$ have non-trivial intersection and do not intersect along a common edge (since then $m_e$ would lie on this edge). Therefore, $T$ can not be an ideal triangulation and this is a contradiction to our initial assumptions.\hfill \textit{QED}.\par
\vspace{\baselineskip}

It follows from the above claim, that if a triangle $\delta$ in $T$ contains the point $m_e$, then this point lies on one of the edges of $\delta$. Such a triangle can either have one, two or three edges which each contain the lift of a point in $E_2$. These triangles will project to a quotient triangle in $\mfaktor{\Phi}{\mathbb{H}}$ of type (IIIa), (IIIb) or (IIIc) (which occurs only for $\Phi=\Gamma_3$), respectively.\par
\vspace{\baselineskip}

\noindent
3. Finally, we show that if $\delta$ contains the lift of an orbifold point in $E_3$, then $\delta$ projects to a monogon containing a single element of $E_3$. As seen above, we can then construct a unique structural arc between this element of $E_3$ and the element of $M$ which lies on the boundary of this monogon.\par
\vspace{\baselineskip}

\noindent
\textbf{Claim:} No edge in $T$ projects to an structural arc in $\mfaktor{\Phi}{\mathbb{H}}$.\par
\vspace{\baselineskip}

\textit{Proof of claim}. Assume that $E$ is an edge of an ideal triangulation $T$ in $\mathbb{H}$, which projects to an structural arc in $\mfaktor{\Phi}{\mathbb{H}}$. Then, $E$ must intersect the centre $c_e$ of a triangle formed be an odd edge $e$ in $P_\Phi$. We define $\phi_e$ to be the side pairing induced by this odd edge. Then, $\phi_e$ is an elliptic involution of order 3 with fixed point $c_e$. In particular, the images of $E$ under $Id$, $\phi_e$ and $\phi_e^{-1}$ are three geodesic which all intersect at $c_e$. Since $T$ is invariant under $\Phi$, all of the images of $E$ under $\Phi$ (and therefore under $Id$, $\phi_e$ and $\phi_e^{-1}$) must be edges in $T$. See Fig.~\ref{quottri}(a). However, we have multiple geodesics intersecting inside $\mathbb{H}$ and therefore, $T$ can not be an ideal triangulation and this is a contradiction to our initial assumptions.\hfill \textit{QED}.\par
\vspace{\baselineskip}

\begin{figure}[htb]
\centering
\begin{subfigure}[b]{0.8\textwidth}

\centering
\begin{tikzpicture}[scale=1.2]

\begin{scope}
    \clip (-1.2,-0.0756) rectangle (1.2,2.1);
    \filldraw (1,0) circle (2pt);
    \filldraw (-1,0) circle (2pt);
    \filldraw (0,0) circle (2pt);
    \filldraw[fill=white] ([xshift=-2pt,yshift=-2pt]0,0.7) rectangle ++(4pt,4pt);
       \draw[dashed,domain=77:180] plot ({0.75*cos(\x)-0.25}, {0.75*sin(\x)});
   \draw[dashed,domain=00:103] plot ({0.75*cos(\x)+0.25}, {0.75*sin(\x)});
      \draw (0,0) circle(1);
      \draw (-1.2,0) to (1.2,0);
      \draw[very thick] (0,0) to (0,0.625);
         \draw[very thick] (0,0.775) to (0,2.1);
         \end{scope}
         
\draw[dashed] (1.75,-0.5) to (1.75,2.6);       
        
        \begin{scope}
        \clip (2.3,-0.0756) rectangle (4.7,2.1);
    \filldraw (3.5,0) circle (2pt);
    \filldraw (2.5,0) circle (2pt);
    \filldraw (4.5,0) circle (2pt);
    \filldraw[fill=white] ([xshift=-2pt,yshift=-2pt]3.5,0.7) rectangle ++(4pt,4pt);
       \draw[very thick,domain=115.5:180] plot ({0.75*cos(\x)+3.75}, {0.75*sin(\x)});
      \draw[very thick,domain=00:104] plot ({0.75*cos(\x)+3.75}, {0.75*sin(\x)}); 
   \draw[very thick,domain=00:64.5] plot ({0.75*cos(\x)+3.25}, {0.75*sin(\x)});
   \draw[very thick,domain=76:180] plot ({0.75*cos(\x)+3.25}, {0.75*sin(\x)});
      \draw (3.5,0) circle(1);
      \draw (2.3,0) to (4.7,0);
      \draw (4,0) circle (0.5);
      \draw (3,0) circle (0.5);
      \draw[very thick] (3.5,0) to (3.5,0.625);
         \draw[very thick] (3.5,0.775) to (3.5,2.1);
         \end{scope}
    \end{tikzpicture} 
    \caption{A geodesic line (bold) passing through the point $c_e$ (left), and its images under $Id$, $\phi_e$ and $\phi_e^{-1}$ (right).}
    \label{rop}
  
     \end{subfigure}\vskip\baselineskip
     
\begin{subfigure}[b]{0.8\textwidth}
\centering
\begin{tikzpicture}[scale=1.2]
\begin{scope}
    \clip (2.3,-0.0756) rectangle (4.7,2.1);
    \filldraw (4.5,0) circle (2pt);
    \filldraw (2.5,0) circle (2pt);
    \filldraw (3.5,0) circle (2pt);
    \filldraw[fill=white] ([xshift=-2pt,yshift=-2pt]3.5,0.7) rectangle ++(4pt,4pt);
       \draw[dashed,domain=77:180] plot ({0.75*cos(\x)+3.25}, {0.75*sin(\x)});
   \draw[dashed,domain=00:103] plot ({0.75*cos(\x)+3.75}, {0.75*sin(\x)});\draw[very thick,domain=00:180] plot ({0.5*cos(\x)+3.5}, {0.5*sin(\x)});
   \draw[domain=00:180] plot ({0.5*cos(\x)+4}, {0.5*sin(\x)});
   \draw[ domain=00:180] plot ({0.5*cos(\x)+3}, {0.5*sin(\x)});
     \draw[ domain=00:180] plot ({cos(\x)+3.5}, {sin(\x)});
      \draw[dashed]  (3.5,0) to (3.5,0.64);
      \draw (2.3,0) to (4.7,0);
      \draw[very thick] (3,0) to (3,2.1);
         \draw[very thick] (4,0) to (4,2.1);
         \end{scope}
         
         \draw[dashed] (1.75,-0.5) to (1.75,2.6);       
        
        \begin{scope}
         \clip (-1.2,-0.0756) rectangle (1.2,2.1);
    \filldraw (1,0) circle (2pt);
    \filldraw (-1,0) circle (2pt);
    \filldraw[fill=white] ([xshift=-2pt,yshift=-2pt]0,0.7) rectangle ++(4pt,4pt);
       \draw[dashed,domain=77:180] plot ({0.75*cos(\x)-0.25}, {0.75*sin(\x)});
        \draw[ domain=00:180] plot ({cos(\x)}, {sin(\x)});
   \draw[dashed,domain=00:103] plot ({0.75*cos(\x)+0.25}, {0.75*sin(\x)});
      \draw (-1.2,0) to (1.2,0);
      \draw[very thick] (-0.5,0.705) to (-0.5,2.1);
         \draw[very thick] (0.5,0.705) to (0.5,2.1);
         \end{scope}
\end{tikzpicture}
\caption{A pair of geodesic rays (left), which form a triangle under the action of $Id$, $\phi_e$ and $\phi^{-1}_e$ (right).} 
\label{ra}
   
\end{subfigure}
\caption{Examples of edges and their images under the actions of $Id$, $\phi_e$ and $\phi^{-1}_e$, where $\phi_e$ is an elliptic involution of order 3 with fixed point $c_e$. The corresponding odd edge $e$ is also shown for structure.}
 \label{quottri}
     \end{figure}
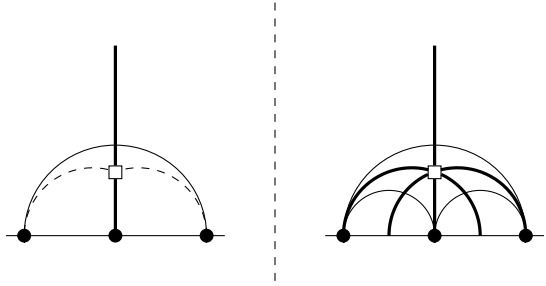
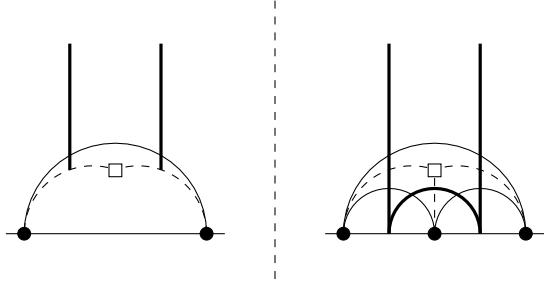

Following this claim, the point $c_e$ must lie in the interior of some triangle $\delta$ in $T$. The elliptic involution about $c_e$ will split $\mathbb{H}$ into three different regions, each containing a vertex of $\delta$. Therefore, the projection of $\delta$ on to $\mfaktor{\Phi}{\mathbb{H}}$ will be a monogon containing a single orbifold point with cone angle $\frac{2\pi}{3}$.
\end{proof}

Let $\lambda$ be a path on $\mfaktor{\Phi}{\mathbb{H}}$ and $\delta$ be a triangle in a quotient triangulation $T_\Phi$, which $\lambda$ passes through. Then we will say that $\lambda$ cuts $\delta$ to form a left (or right) triangle if, once having removed the structural arcs, the lift of $\lambda$ cuts the lift of $\delta$ to form a left (or right) triangle. We then derive  the cutting sequence of $\lambda$ with $T_\Phi$ in the usual sense, which we denote $(\lambda,T_\Phi)$. Here the space $\mfaktor{\Phi}{\mathbb{H}}$ is implied by the quotient triangulation. Obviously, the cutting sequence $(\lambda,T_\Phi)$ will be equivalent to the cutting sequence $(\overline{\lambda},\overline{T})$, where $\overline{\lambda}$ is the lift of $\lambda$ in $\mathbb{H}$ and $\overline{T}$ is the lift of $T_\Phi$ with structural arcs removed (which will be an ideal triangulation).

Note that since we remove the structural arcs when defining the left and right triangles, we do not necessarily need a quotient triangulation to take a cutting sequence and a quotient triangulation with structural arcs removed will be sufficient. However, since there is a unique way to construct these structural arcs, we will equate these two objects anyway.

Since $\mfaktor{\Gamma_0(n)}{\mathbb{H}}$ is equivalent to $\sfaktor{P_n}{\sim}$ (a special polygon for $\Gamma_0(n)$ quotient the side identifications), it follows that the projection of $\frac{1}{d}\mathcal{F}$ onto $\mfaktor{\Gamma_0(n)}{\mathbb{H}}$ for $d\mid{n}$, is equivalent to $\sfaktor{T_{\{d,n\}}}{\sim}$, the copy of $P_n$ with decoration induced by ${\frac{1}{d}\mathcal{F}}$, quotient the side identifications. Using this information and Lemma \ref{Triang}, we get the following theorem.

\begin{theorem}\label{orbi}
For every geodesic ray $\widetilde{\zeta}$ in $\mathbb{H}$ starting at the $y$-axis $I$  with endpoint $\alpha>0$, there is a canonical projection $\zeta$ onto $\mfaktor{\Gamma_0(n)}{\mathbb{H}}$ such that $(\widetilde{\zeta},\frac{1}{d}\mathcal{F})=(\zeta,\sfaktor{T_{\{d,n\}}}{\sim})$, for all $d\mid{n}$.
\end{theorem}

\begin{proof} 
By Lemma \ref{Triang}, we see that the projection of $\frac{1}{d}\mathcal{F}$ on $\mfaktor{\Gamma_0(n)}{\mathbb{H}}$ decomposes $\mfaktor{\Gamma_0(n)}{\mathbb{H}}$ into triangles of type (I)-(IIIb) or into monogons containing a single element of $E_3$, for all $d\mid{n}$. Since the $y$-axis $I$ is an edge in $\frac{1}{d}\mathcal{F}$ for all $d\in\mathbb{N}$, the projection $\zeta$ of $\widetilde{\zeta}$ in $\mfaktor{\Gamma_0(n)}{\mathbb{H}}$ is  unique and has a well defined starting edge and direction of departure, for all $\sfaktor{T_{\{d,n\}}}{\sim}$. It follows that, since the cutting sequence in $\mfaktor{\Gamma_0(n)}{\mathbb{H}}$ of $\zeta$ with $\sfaktor{T_{\{d,n\}}}{\sim}$ is independent of structural arcs, the cutting sequence in $\mfaktor{\Gamma_0(n)}{\mathbb{H}}$ of $\zeta$ with $\sfaktor{T_{\{d,n\}}}{\sim}$ is both well-defined and equivalent to the cutting sequence in $\mathbb{H}$ of $\widetilde{\zeta}$ with $\frac{1}{d}\mathcal{F}$ for all $d\mid{n}$. 
\end{proof}

\subsection{Closed Curves as Cutting Sequences}

For an arbitrary infinite sequence $\{a_i\}_{i\in\mathbb{N}}$, we will say the sequence is \textit{periodic} if there exists an $s\in\mathbb{N}$ such that $a_i=a_{s+i}$ for all $i\in\mathbb{N}$. We will write this sequence as $\{a_i\}_{i\in\mathbb{N}}=a_1,a_2,\ldots,a_s,a_1,a_2,\ldots,a_s,\ldots=\overline{a_1,a_2,\ldots,a_s}$ and refer to $s$ as a \textit{period} of the sequence. We use this to define the following types of continued fraction.

\begin{defn}\label{esp}
\begin{enumerate}
\item A \textit{strictly periodic} continued fraction, is any continued fractions with partial quotient expansion of the form $[\overline{a_0;a_1,\ldots,a_{s-1}}]$ or $[0;\overline{a_1,\ldots,a_s}]$. We refer to the set of all strictly periodic continued fraction as \textit{SP}.

\item An \textit{essentially periodic} continued fraction, is any continued fractions with partial quotient expansion of the form $[a_0;\overline{a_1,\ldots,a_{s}}]$ with $a_0\leq{a_s}$ or $[0;a_1,\overline{a_2,\ldots,a_{s+1}}]$ with $a_1\leq{a_s}$. We refer to the set of all essentially periodic continued fraction as \textit{ESP}.

\item An \textit{eventually periodic} continued fraction, is any continued fraction with partial quotient expansion of the form $[b_0;\ldots,b_r,\overline{a_1,\ldots,a_{s}}]$ where $r\in\mathbb{N}\cup{\{0\}}$. We refer to the set of all eventually periodic continued fraction as \textit{EVP}.

\end{enumerate}
\end{defn}

\begin{remark} An immediate consequence of these definitions is that \textit{SP}$\subset$\textit{ESP}$\subset$\textit{EVP}. We refer to the set of all positive strictly periodic continued fractions as \textit{SP}$^+$. Similarly, we refer to the sets of positive essentially periodic continued fractions and positive eventually periodic continued fractions as \textit{ESP}$^+$ and \textit{EVP}$^+$, respectively.\end{remark}

We can equivalently define essentially periodic continued fractions to be continued fractions of the form $[\overline{a_0;a_1,\ldots,a_{s}}]$ or $[0;\overline{a_1,\ldots,a_s}]$, where $a_0\in{\mathbb{Z}}$ and $a_i\in\mathbb{N}\cup\{0\}$. We will simplify this alternative definition to the case where $a_i\in\mathbb{N}$ for $0<i<s$ but $a_n\in\mathbb{N}\cup\{0\}$, since we can  always find such a presentation. In the case that $a_s\neq{0}$, we just get a strictly periodic continued fraction. In the case that $a_s=0$, we concatenate terms either side of the $0$ term. Explicitly, we have the following:
\begin{align*}
[\overline{a_0;a_1,\ldots,a_{s-1},0}]= & [a_0;a_1,\ldots,a_{s-1},0,a_0,a_1,\ldots,a_{s-1},0,a_0,\ldots]\\
 = & [a_0;a_1,\ldots,a_{s-1}+a_0,a_1,\ldots,a_{s-1}+a_0,\ldots]\\
 = & [a_0;\overline{a_1,\ldots,a_{s-1}+a_0}]\\
[0;\overline{a_1,\ldots,a_{s-1},0}]=&[0;a_1,\ldots,a_{s-1},0,a_1,\ldots,a_{s-1},0,a_1,\ldots]\\
= & [0;a_1,\ldots,a_{s-1}+a_1,a_2,\ldots,a_{s-1}+a_1,\ldots]\\
= & [0;a_1,\overline{a_2,\ldots,a_{s-1}+a_1}]
\end{align*}

\noindent
In terms of fans, a zero term corresponds to having an empty fan. Therefore, since the fans either side are of the same type, they both collapse into one bigger fan. It is occasionally useful to have a place holder fan of size zero, particularly when dealing with closed curves on our orbifold.

\begin{remark} We can similarly define the notion of being strictly, essentially or eventually periodic for all sequences of numbers and in particular for cutting sequences.
\end{remark}

For all eventually periodic continued fractions, we can write the periodic part with even period. To do this we simply take two copies of the period and take this to be our new period i.e. $[b_0;b_1,\ldots,b_r,\overline{a_1,\ldots,a_s}]=[b_0;b_1,\ldots,b_r,\overline{a_1,\ldots,a_s,a_1,\ldots,a_s}]=[b_0;b_1,\ldots,b_r,\overline{a'_1,\ldots,a'_{2s}}]$, where $a'_j=a_i$ for $j\equiv{i}$ $mod(s)$. For the rest of the paper, we will write all eventually periodic continued fractions with an even period. Taking the period to be even ensure that when we take the associated cutting sequence, the initial term and the final term of the period will correspond to different letters i.e. $\overline{L^{a_0}\cdots{R^{a_{2s}}}}$. This ensures that the parity of the cutting sequence will be nice, i.e. when we take multiple copies of the period, every term in the sequence will alternate. The following example emphasises why we take even periods:
$$[\overline{2;1,1}]=[\overline{2;1,1,2,1,1}]=\eta(\overline{L^2RLR^2LR})\neq{\eta(\overline{L^2RL})}$$
$$ [\overline{2;1,1,0}]=\eta(\overline{L^2RLR^0})=\eta(\overline{L^2RLR^0L^2RLR^0})=\eta(\overline{L^2RLL^2RL})=\eta(\overline{L^2RL}) $$

\noindent
We can also ensure that the finite prefix of an eventually periodic continued fraction has an even number of terms by shifting the period by a single term, if necessary. In other words, if we had the continued fraction expansion $[b_0;b_1,\ldots,b_{2r}\overline{a_1,\ldots,a_{2s}}]$ then this is equivalent to $[b_0;b_1,\ldots,b_{2r},a_1,\overline{a_2,\ldots,a_{2s},a_1}]$.

Whilst \textit{ESP}$^+$ is perhaps an unnatural object in a typical number theory setting, it is a very natural object with regards to the geometric approach. This is emphasised in the following theorem.

\begin{theorem}\label{cc}
Let $\mathcal{O}$ be an orbifold with quotient triangulation $T$. Then a path $\zeta$ relative to a starting edge $E$ in $T$ (excluding structural arcs) is homotopic to a closed curve on $\mathcal{O}$ if and only if $\eta((\zeta,T))\in$\textit{ESP}$^+$.
\end{theorem}

\begin{proof}$(\Rightarrow)$: Let $\zeta$ be a closed curve which passes through the edge $E$ in $\mathcal{O}$. Let $\zeta_{1}$ be the representation of $\zeta$, which starts and ends at $E$ and follows $\zeta$ exactly once. Similarly, define $\zeta_i$ for $i\in\mathbb{N}$, to be the representation of $\zeta$ that starts and ends at $E$ and goes round $\zeta$ exactly $i$ times. By using a zero term as place holder (if required), we can always write $(\zeta_{1},S)$ in the form $\{n_0,n_1,\ldots,n_{2m-1}\}=L^{n_0}R^{n_1}\cdots{R^{n_{2m-1}}}$, where $n_0,n_{2m-1}\in\mathbb{N}\cup\{0\}$ and $n_i\in\mathbb{N}$ for $i\in\{1,\ldots,2m-2\}$. As a result, we can express $(\zeta_i,T)$ in alternating fans, as follows:
\begin{align*}
(\zeta_{i},T) &=\{\underbrace{n_0,n_1,\ldots,n_{2m-1},\ldots,n_0,n_1,\ldots,n_{2m-1}}_{\text{$i$ times}}\} \\
&=\underbrace{L^{n_0}R^{n_1}\cdots{R^{n_{2m-1}}\cdots{L^{n_0}R^{n_1}\cdots{R^{n_{2m-1}}}}}}_{\text{$i$ times}}
\end{align*}

Taking the limit as $i$ tends to infinity, we observe that $(\zeta,T)=\lim_{i\rightarrow{\infty}}(\zeta_i,T)=\{\overline{n_0,n_1,\ldots,n_{2m-1}}\}=\overline{L^{n_0}R^{n_1}\cdots{R^{n_{2m-1}}}}$. We investigate the four cases, which arise from this, based on whether $n_0,n_{2m-1}$ are zero or non-zero:
\begin{itemize}
\item If $n_0\neq{0}$ and $n_{2m-1}\neq{0}$, then the resulting word is reduced and therefore $\eta((\zeta,T))\in{SP^+}\subset{ESP^+}$.

\item If $n_0=n_{2m-1}=0$, then $(\zeta,S)=\overline{L^0R^{n_1}\cdots{R^0}}=\overline{R^{n_1}\cdots{L^{n_{2m-2}}}}$. This is now in reduced form and therefore $\eta((\zeta,T))\in{SP^+}\subset{ESP^+}$.

\item If $n_0=0$ and $n_{2m-1}\neq{0}$, then $(\zeta,T)=\overline{L^0R^{n_1}\cdots{R^{n_{2m-1}}}}=\overline{R^{n_1}\cdots{R^{n_{2m-1}}}}=R^{n_1}\overline{L^{n_2}\cdots{R^{n_1+n_{2m-1}}}}$. This is now in reduced form and therefore $\eta((\zeta,T))\in{ESP^+}$.

\item If $n_0\neq{0}$ and $n_{2m-1}={0}$, then $(\zeta,T)=\overline{L^{n_0}R^{n_1}\cdots{R^{0}}}=\overline{L^{n_0}R^{n_1}\cdots{L^{n_{2m-2}}}}=L^{n_0}\overline{R^{n_1}\cdots{L^{n_0+n_{2m-2}}}}$. This is now in reduced form and therefore $\eta((\zeta,T))\in{ESP^+}$.
\end{itemize}  
Since all possible forms of $(\zeta,T)$ satisfy $\eta((\zeta,T))\in{ESP^+}$, the result follows.\par\vspace{\baselineskip}

%We are concerned with how $\zeta_\alpha$ joins back to itself at the edge $E$. When $\zeta_{\alpha,1}$ intersects $S$, the initial and final fan can either be of the same type i.e both left and right, or of different type.
%If these fans do not match, then $(\zeta_{\alpha,2},S)$ is of the form $\{a_0;a_1,\ldots,a_n,a_0,a_1,\ldots,a_n\}$ or $\{0;a_1,\ldots,a_n,a_1,\ldots,a_n\}$. Taking $(\zeta_\alpha,S)=\lim_{i\rightarrow\infty}(\zeta_{\alpha,i},S)$, we get a (positive) strictly periodic continued fraction. In particular, either $(\alpha,S)=[\overline{a_0;a_1,\ldots,a_{n}}]$ or $(\alpha,S)_{E'}=[0;\overline{a_1,\ldots,a_n}]$).
%In the case that the start and end fans of $\zeta_\alpha$ do match, then these fans collapse into a bigger fan. In particular $(\alpha_2,S)$ is either of the form $[a_0;a_1,\ldots,a_n+a_0,a_1,\ldots,a_n]$ or $[0;a_1,\ldots,a_n+a_1,\ldots,a_n]$. We can rewrite these terms as $[a_0;a_1,\ldots,a_n,0,a_0,a_1,\ldots,a_n]$ or $[0;a_1,\ldots,a_n,0,a_1,\ldots,a_n]$ respectively. This is analogous to saying that we take our last fan to of different type to our start fan, but this fan may be of size $0$. Taking $(\alpha,S)_{E'}=\lim_{i\rightarrow\infty}(\alpha_i,S)_{E'}$, we get a (positive) essentially periodic continued fraction. In particular, $(\alpha,S)_{E'}=[a_0;\overline{a_1,\ldots,a_n,0}]$ or $(\alpha,S)_{E'}=[0;\overline{a_1,\ldots,a_n,0}]$.

$(\Leftarrow)$:  Since $\mathcal{O}$ is finitely triangulated with quotient triangulation $T$, we can define a cutting sequence on it. We denote the set of non-structural arcs in $T$ as $\mathcal{E}$. Every edge $E\in\mathcal{E}$ will be an edge of at most two triangles in the triangulation $T$, which we arbitrarily label $\tau^+_E$ and $\tau^-_E$. For each of these triangles, we can approach $E$ from at most two directions, that is, either via a left triangle or a right triangle. Note that for $E$ an edge from an element of $M$ to an element of $E_2$, this edge will have a single direction of approach. We define $\overline{\mathcal{E}}$ to be the set of all possible directions of approach for all edges in $\mathcal{O}$. We can then think of the abstract set of all theoretical directions of approach as the Cartesian product $\mathcal{E}\times\{+,-\}\times\{L,R\}$, where the set $\{+,-\}$ represents the choice of approaching an edge $E$ via $\tau^+_E$ or $\tau^-_E$ and the set $\{L,R\}$ represents the choice of approaching an edge $E$ via left or right triangles. Since $\mathcal{E}$, $\{+,-\}$ and $\{L,R\}$ are all finite sets, $\overline{\mathcal{E}}\subset\mathcal{E}\times\{+,-\}\times\{L,R\}$ is also a finite set. See Fig.~\ref{doa} for a pictorial representation of direction of approach.

Let $\overline{\alpha}$ be a positive essentially periodic continued fraction with one of two forms:

\renewcommand{\labelenumi}{(\roman{enumi})}
 \begin{enumerate}[wide=15pt, widest=99,leftmargin=45pt, labelsep=*] 
\item $\overline{\alpha}=[\overline{a_0;a_1,\ldots,a_{2s-1}}]$ with $a_{s-1},a_{2s-1}\in\mathbb{N}\cup\{0\}$ and $a_i\in\mathbb{N}$ otherwise.
\item $\overline{\alpha}=[0;\overline{a_1,\ldots,a_{2s}}]$ with $a_s,a_{2s}\in\mathbb{N}\cup\{0\}$ and $a_i\in\mathbb{N}$ otherwise.
\end{enumerate}
 \begin{figure}[ht]
  \includegraphics[width=0.8\linewidth]{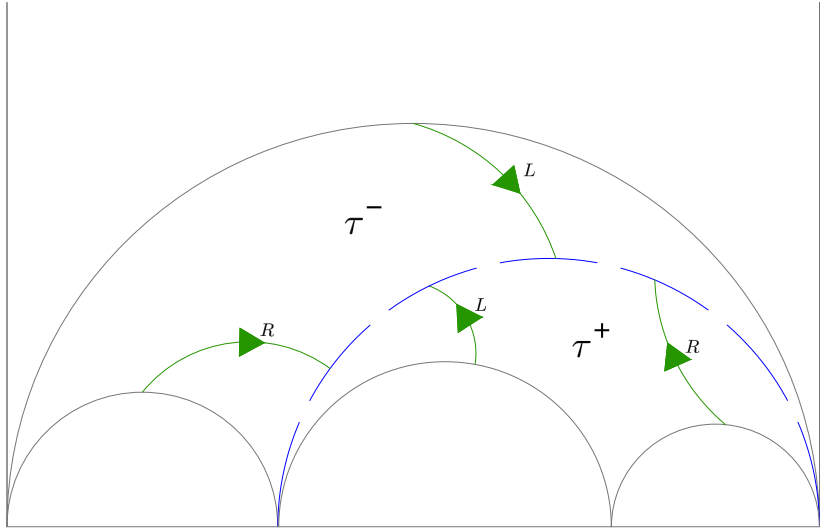} 
  \caption{An image depicting the four possible directions of approach to the blue edge $E$.}
  \label{doa}
\end{figure}

\noindent
For case (i),  we can write $\eta^{-1}(\overline{\alpha})=\overline{L^{a_0}R^{a_1}\cdots{R^{a_{2s-1}}}}$. Similarly, for case (ii), we can write $\eta^{-1}(\overline{\alpha})=\overline{R^{a_1}L^{a_2}\cdots{L^{a_{2s}}}}$. We define the object $\overline{\alpha_1}$ as follows:

\begin{align*}
\overline{\alpha_1}=
\begin{cases} [a_0;a_1,\ldots,a_{2s-1}] &\text{for case (i)}  \\
 [0;a_1,a_2,\ldots,a_{2s}] &\text{for case (ii)} 
\end{cases}
\end{align*}

For both of the above values of $\overline{\alpha_1}$, $\eta^{-1}(\overline{\alpha_1})$ starts and ends with different letters, since $\overline\alpha_1$ is of even length. 

Given a starting edge $E_0$ with chosen direction of departure, we can find a homotopy class of geodesic rays $[\zeta]$, such that $\eta((\zeta,T))=\overline{\alpha}$, for all $\zeta\in[\zeta]$. We define $\zeta_{1}$ to be a geodesic path (unique up to homotopy), which starts at $E_0$ (with the same direction of departure as $\zeta$), terminates at some edge $E_1$ and satisfies $\eta((\zeta_{1},T))=\overline{\alpha_1}$. We similarly define $\zeta_i$ to be a  geodesic path (up to homotopy) in $\mathcal{O}$ relative to $E_{i-1}$, with direction of departure opposite to the direction of approach for $\zeta_{i-1}$, such that $(\zeta_i,T)=\overline{\alpha_1}$. Due to the fact we took $\overline{\alpha_1}$ with even length, $\zeta_i$ approaches $E_i$ via the same type of triangles for every $i\in\mathbb{N}$, i.e. if $\zeta_1$ approaches $E_1$ via left triangles, then every $\zeta_i$ approaches $E_i$ via left triangles. Note that this fan may in fact be empty but this does not change the fact that the cutting sequence consists solely of alternating terms. Since the cutting sequences $(\zeta,T)$ and $(\bigcup^{\infty}_{i=1}\zeta_i,T)$ (with canonical ordering) are equivalent and $\zeta$ and $\zeta_1$ have the same starting edge, we can homotope the collection of $\zeta_i$, $\bigcup^{\infty}_{i=1}\zeta_i$, to look like $\zeta$. Then, the set of edges (with the direction of approach) which are the endpoints for each $\zeta_i$,  $\bigcup^{\infty}_{i=1}E_i$ is a subset of $\overline{\mathcal{E}}$.
Since $i$ runs from $1$ to $\infty$ and $\overline{\mathcal{E}}$ is finite, it follows from the pigeon-hole principle that there exist $j,k\in\mathbb{N}$ such that $E_j=E_{j+k}$, with the same direction of approach and the same type of approach via left/right triangles. 

As seen above, for each $E_{i}$, we can uniquely define up to homotopy the path $\zeta_i$. By the same argument, for each $E_{i}$ we can uniquely define up to homotopy the path $\zeta_{i-1}$. %This path is unique path (up to homotopy) with opposite direction of departure to  $\zeta_i$ and cutting sequence with $\mathcal{F}$: $(i)$ $L^{a_{2n-1}}R^{a_{2n-2}}\cdots{R^{a_0}}$ or $(ii)$ $R^{a_{2n}}L^{a_{2n-1}}\cdots{L^{a_1}}$.
Hence, if $E_{j+k}=E_j$, then $\zeta_{j-1}$ and $\zeta_{j+k-1}$ are homotopic and $E_{j-1}=E_{j+k-1}$. Using iteration, we recover that for such a $k$, $E_0=E_k$. For each $l,l'\in\mathbb{N}$ with $l\equiv{l'}$ mod $k$, the paths $\zeta_l$ and $\zeta_{l'}$ are homotopic. We can homotope all such paths $\zeta_l$ and $\zeta_{l'}$  to be concurrent and since $E_0=E_k$, we recover a closed curve. Since $\zeta$ is homotopic to $\bigcup^{\infty}_{i=1}\zeta_i$, $\zeta$ is therefore homotopic to a closed curve.
\end{proof}

It is worth noting that this technique works for any type of periodicity, that is, if $\overline\alpha$ has a periodic tail, the representation $\zeta_\alpha$ of $\overline\alpha$, relative to some quotient triangulation of $\mathcal{O}$, is homotopic to a closed curve. If $\overline\alpha$ is essentially periodic, $\zeta_\alpha$ is homotopic to a closed curve, as above. If $\overline\alpha$ is eventually periodic, $\zeta_\alpha$ can be decomposed into a finite path which joins onto an infinite path homotopic to a closed curve. As a result, we can show that $ESP^+$ and $EVP^+$ are both closed classes under multiplication by any rational number. Below we provide the statement for $ESP^+$, but the proof that $EVP^+$ is closed under rational multiplication follows trivially. 

\begin{cor}[Theorem \ref{cc}]\label{rcc} The $\overline{n}$ map, maps $ESP^+$ to $ESP^+$. In particular, for all $q\in\mathbb{Q}_{>0}$ and $\overline{\alpha}\in{ESP^+}$, we have $\overline{q\alpha}\in{ESP^+}$. 
\end{cor}

\begin{proof} Let $\zeta_\alpha$ be a geodesic ray, starting at the $y$-axis and terminating at some point $\alpha>0$ with $\overline{\alpha}\in{ESP^+}$. Then $\overline{\zeta_\alpha}$ is homotopic to a closed curve in $\sfaktor{T_{\{1,n\}}}{\sim}$. The map $\overline{n}:(\overline{\zeta_\alpha},\sfaktor{T_{\{1,n\}}}{\sim})\to{(\overline{\zeta_\alpha},\sfaktor{T_{\{n,n\}}}{\sim})}$ maps $\overline{\zeta_\alpha}$ to itself. As such, $\zeta_\alpha$ is a closed curve in $\sfaktor{T_{\{n,n\}}}{\sim}$ and $\overline{n\alpha}=(\zeta_\alpha,\sfaktor{T_{\{n,n\}}}{\sim})$ defines a positive essentially periodic continued fraction by Theorem \ref{cc}.

We can similarly define the map ${\overline{n}}^{-1}:(\zeta_\alpha,\sfaktor{T_{\{n,n\}}}{\sim})\to{(\zeta_\alpha,\sfaktor{T_{\{1,n\}}}{\sim})}$ to recover $\overline{\frac{\alpha}{n}}=(\zeta_\alpha,\sfaktor{T_{\{1,n\}}}{\sim})$ from $\overline{\alpha}=(\zeta_\alpha,\sfaktor{T_{\{n,n\}}}{\sim})$. Let $q=\frac{r}{s}$. Since $\overline{r\alpha}\in{ESP^+}$ and $\overline{\frac{\alpha}{s}}\in{ESP^+}$, $\overline{q\alpha}=\overline{\frac{r}{s}\alpha}\in{ESP^+}$.
\end{proof}

\subsubsection{Convergents of Essentially Periodic Continued Fractions}

The property of $\zeta_\alpha$ being homotopic to a closed curve on some orbifold $\mathcal{O}$ can also be used to prove some interesting facts regarding the convergent denominators/numerators of both strictly periodic and essentially periodic continued fractions. Here we take our orbifold to be $\mfaktor{\Gamma_0(n)}{\mathbb{H}}$ for some $n\in\mathbb{N}$. Using the link between convergents and common vertices of the fans in $(\zeta_\alpha,\mathcal{F})$, and Theorem \ref{cc}, we show that for any natural number $n$ and $\overline{\alpha}$ any strictly periodic continued fraction, $\alpha$ has infinitely many convergent denominators and numerators divisible by $n$. We also show for any natural number $n$ and any essentially periodic continued fraction $\overline{\alpha}$, that:

\begin{itemize}
\item For $\alpha>1$, there are infinitely many convergent denominators of $\overline{\alpha}$ which are divisible by $n$.
 \item For $\alpha<1$, there are infinitely many convergent numerators of $\overline{\alpha}$ which are divisible by $n$.
 \end{itemize}

\begin{theorem}\label{conden}
\tolerance=800
Let $\overline\alpha\in{SP^+}$, then for every $n\in\mathbb{N}$ there are infinitely many convergent denominators $q_k$ of $\overline\alpha$ such that $n\mid{q_k}$.

Let $\overline\alpha\in{ESP^+}$ with $\alpha{>1}$, then for every $n\in\mathbb{N}$ there are infinitely many convergent denominators $q_k$ of $\overline\alpha$ such that $n\mid{q_k}$.
\end{theorem}

\begin{proof} %If we take $\zeta_\alpha$ to be the geodesic ray starting at the $y$-axis and terminating at the point $\alpha\in\mathbb{R}_{>0}$, then the convergent of $\overline{\alpha}$ given by $\frac{p_k}{q_k}=[a_0:a_1,\ldots,a_k]$ is the fixed vertex in the $k+1^{\text{th}}$ fan for $(\zeta_\alpha,\mathcal{F})$ (treating $a_0$ as the $0^{\text{th}}$ fan). The convergents of $\alpha$ are therefore given by the fixed vertices in the fans of $(\zeta_\alpha,\mathcal{F})$.

%For $A=\begin{psmallmatrix} a & b\\ mc & d \end{psmallmatrix}\in\Gamma_0(n)$, $A\cdot0=\frac{b}{d}$ and $A\cdot\infty=\frac{a}{mc}$, where $gcd(mc,d)=gcd(mc,a)=1$ (since $|ad-mcb|=1$). Therefore, there is a line in $\mathcal{F}$ and $\frac{1}{m}\mathcal{F}$ between $\frac{b}{d}$ and $\frac{a}{mc}$. 

Let $\zeta_\alpha$ be a geodesic ray, starting at the $y$-axis $I$ and terminating at some point $\alpha\in\mathbb{R}_{>0}$, with $\overline{\alpha}\in\text{ESP}^+$. We can take $\overline{\zeta_\alpha}$ to be the image of $\zeta_\alpha$ in the orbifold $\mfaktor{\Gamma_0(n)}{\mathbb{H}}$, with $(\overline{\zeta_\alpha},\sfaktor{T_{\{1,n\}}}{\sim})=(\zeta_\alpha,\mathcal{F})$. By Theorem $\ref{cc}$, it follows that $\overline{\zeta_\alpha}$ is homotopic to a closed curve in $\mfaktor{\Gamma_0(n)}{\mathbb{H}}$. We can take $E$ to be the image of $I$ in $\mfaktor{\Gamma_0(n)}{\mathbb{H}}$ and as a result $\overline{\zeta_\alpha}$ will start at the edge $E$. Since $\overline{\zeta_\alpha}$ is homotopic to a closed curve, it follows that $\overline{\zeta_\alpha}$  intersects this edge $E$ infinitely often. 

Since $P_n$ is a fundamental domain for $\Gamma_0(n)$ and $I$ maps to $E$, the edge $E$ lifts to the set of all edges $\Gamma_0(n)({I}):=\{\rho({I}):\rho\in\Gamma_0(n)\}$ in $\mathbb{H}$.  From section \ref{struc}, we can explicitly describe $\Gamma_0(n)({I})$ as the set of edges with vertices $\frac{a}{nc}$ and $\frac{b}{d}$, with $\begin{psmallmatrix} a & b\\ nc & d \end{psmallmatrix}\in\Gamma_0(n)$. For $\rho=\begin{psmallmatrix} a & b\\ nc & d \end{psmallmatrix}\in\Gamma_0(n)$, $\rho(\infty)=\frac{a}{nc}$ and $\rho(0)=\frac{b}{d}$.

Because $\overline{\zeta_\alpha}$ intersects $E$ infinitely often, $\zeta_\alpha$ intersects $\Gamma_0(n)({I})$ infinitely often. If $\zeta_\alpha$ intersects $\rho(I)$, for some $\rho\in\Gamma_0(n)$, then due to the fact that every point $\rho(\infty)$ behaves locally like $\infty$ in $\mathcal{F}$ and $\frac{1}{n}\mathcal{F}$ (see section \ref{struc}), $\rho(\infty)$ will be a convergent of $\overline{\alpha}$ if and only if $\zeta_\alpha$ approaches the edge $\rho(I)$ via a left triangle or departs the edge $\rho(I)$ via a left triangle. 

We consider four different cases for the possible values of $(\zeta_\alpha,\mathcal{F})$ for $\overline{\alpha}\in{ESP^+}$:

 \renewcommand{\labelenumi}{(\roman{enumi})}
 \begin{enumerate}[wide=15pt, widest=99,leftmargin=45pt, labelsep=*] 
 
 \item $\overline{\alpha}\in{SP^+}$ $\alpha>1$: then $(\zeta_\alpha,\mathcal{F})=\overline{L^{a_0}R^{a_1}\cdots{R^{a_{2s-1}}}}$ and $a_i\in\mathbb{N}$.
 
 \item $\overline{\alpha}\in{SP^+}$ $\alpha<1$: then $(\zeta_\alpha,\mathcal{F})=\overline{R^{a_1}\cdots{L^{a_{2s}}}}$ and $a_i\in\mathbb{N}$.
 
 \item $\overline{\alpha}\in{ESP^+}$ $\alpha>1$: then $(\zeta_\alpha,\mathcal{F})=L^{a_0}\overline{R^{a_1}\cdots{L^{a_{2s}}}}=\overline{L^{a_0}R^{a_1}\cdots{L^{a_{2m}-a_0}}R^0}$ with $0<a_0\leq{a_{2s}}$ and $a_i\in\mathbb{N}$.
 
 \item $\overline{\alpha}\in{ESP^+}$ $\alpha<1$: then $(\zeta_\alpha,\mathcal{F})=R^{a_1}\overline{L^{a_2}\cdots{R^{a_{2s-1}}}}=\overline{R^{a_1}L^{a_2}\cdots{R^{a_{2s-1}-{a_1}}}L^0}$ with $0<a_1\leq{a_{2s-1}}$  and $a_i\in\mathbb{N}$.
 \end{enumerate}

From the construction in Theorem \ref{cc}, we can guarantee that we approach the edge $E$ via the last term in the period and depart $E$ via the first term in our period. Therefore, we can guarantee that $\zeta_\alpha$ approaches or departs infinitely many edges $\rho(I)$ via the last term in our period or the first term in our period respectively. For (i)-(iii) we can always write the period to start or end with non-empty left fan and as such, there exist infinitely many elements $\rho\in\Gamma_0(n)$ such that $\rho(\infty)$ will be a convergent of $\overline{\alpha}$. For case (iv), we can not guarantee such a result. Since $\alpha\neq{\infty}$, $\zeta_\alpha$ can only intersect finitely many edges $\rho(I)$, where $\rho(\infty)=\infty$. In particular, for cases (i)-(iii), $\overline{\alpha}$ has infinitely many convergents of the form $\frac{a}{nc}$ for $a,c,\in\mathbb{N}$.
%We define $\zeta_1$ to be the path which goes round the closed curve exactly once, we can collect the information of all the side-pairings $\zeta_1$ passes through in order. Each of these side pairings has a transformation associated to it and when we pre-compose each of the transformations, we get an element $\rho\in\Gamma_0(p)$. This transformation maps our starting edge $0$ to $\infty$ in $\mathbb{H}$, to the edge at which $\zeta$ begins to repeat. $\rho^k$ then maps $\zeta$ to itself. 
%This line is our starting edge for our geodesic ray, and therefore for $\alpha$ our path representation of $\overline\alpha$ in $T_m$, we come back to this edge infinitely often. t.
%For strictly periodic continued fractions, we change fans at this edge either side of this edge. Each vertex of this edge is a convergent of $\alpha$ and by the above statement one of these vertices is of the form $\frac{a}{mc}$ with $gcd(a,mc)=1$. Repeating this procedure for all $\alpha_i$ $i\in\mathbb{N}$ we recover the required result.
%For essentially periodic continued fractions greater than one, the first term is a non-empty left fan. As a result, the image of $\infty$ under the above transformation is a fixed point in the next fan and as such is a convergent of $\overline\alpha$. The rest follows as in the strictly periodic case.
\end{proof}

Unfortunately, this result does not hold in general for $\overline{\alpha}\in\text{ESP}^+$ with $0<\alpha{<1}$. When we write the corresponding cutting sequence, the period both starts and ends with a non-empty right fan and as such we can not guarantee a convergent of the form $\frac{a}{nc}$. An example of such a result is that for $\overline\alpha=[0;1,\overline{1,1,2}]$, $5$ does not divide any of the convergent denominators.
%
%\begin{remark} The first twenty convergents for $[0;1,\overline{1,1,2}]$ are:
%
%$$ 0, 1, \frac{1}{2}, \frac{2}{3}, \frac{5}{8}, \frac{7}{11}, \frac{12}{19}, \frac{31}{49},\frac{43}{68}, \frac{74}{117}, \frac{191}{302}, \frac{265}{419}, \frac{456}{721}, \frac{1177}{1861}, \frac{1633}{2582}, \frac{2810}{4443}, \frac{7253}{11468}, \frac{10063}{15911}, \frac{17316}{27379}, \frac{44695}{70669}
%$$
%\end{remark}
%
Since, for $0<\alpha<1$ with $\overline{\alpha}\in{ESP}^+$, $\frac{1}{\alpha}>1$ and $\overline{\frac{1}{\alpha}}\in{ESP^+}$, we get the following corollary. %Instead, by the inverse duality of real numbers less than one with real numbers greater than one, we can say that for $\overline\alpha\in{SP^+}\cup{ESP^+}_{<1}$ and for every $m\in\mathbb{N}$ there are infinitely many convergent numerators $p_k$ of $\overline\alpha$ such that $m\mid{p_k}$. 
\begin{cor}\label{connum} Let $\overline\alpha\in{SP^+}$, then for every $n\in\mathbb{N}$ there are infinitely many convergent numerators $p_k$ of $\overline\alpha$ such that $n\mid{p_k}$.

Let $\overline\alpha\in{ESP^+}$ with $0<\alpha{<1}$, then for every $n\in\mathbb{N}$ there are infinitely many convergent numerators $p_k$ of $\overline\alpha$ such that $n\mid{p_k}$.
\end{cor}

\subsection{Eventually Periodic Continued Fractions in Relation to Essentially Periodic Continued Fractions}

\begin{theorem}\label{evp}
Let $\beta$ be in $EVP^+\setminus{ESP^+}$. Then for every $n\in\mathbb{N}$ there exists an $a,k\in\mathbb{N}$ and $\alpha\in{ESP}^+$ such that $\overline{mn^k\beta}=ma+\overline{m\alpha}$. In particular, $n^k\beta$ will be of the form $[a'_0;\overline{a_1,\ldots,a_s}]$ with $a'_0>a_s$. %That is,  and there exists a geodesic ray $\zeta_\alpha$  in $ESP^+$, such that for any $m\in\mathbb{N}$, $(\zeta_\beta,\frac{1}{mn^k}\mathcal{F})=ma+(\zeta_\alpha,\frac{1}{mn^k}\mathcal{F})$ for $a\in\mathbb{N}$ some constant.
\end{theorem}

 \begin{figure}[htbp]
\centering
\begin{subfigure}{.8\textwidth}
  \centering
  \includegraphics[width=1\linewidth]{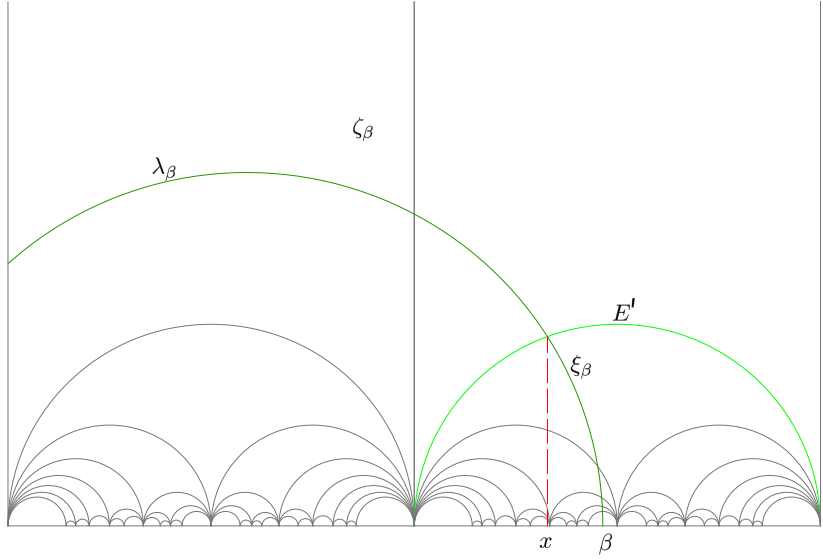}
  \caption{A diagram showing $\lambda_\beta$ and $\xi_{\beta}$ as sub-paths of $\zeta_\beta$.}
  \label{sub1}
\end{subfigure}%
\\
\begin{subfigure}{.8\textwidth}
  \centering
  \includegraphics[width=1\linewidth]{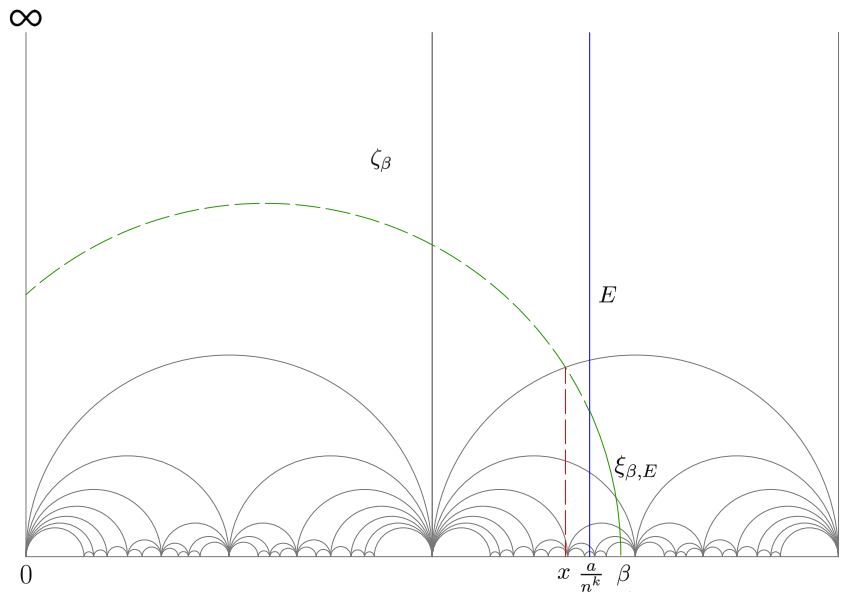}
  \caption{A diagram showing $\xi_{\beta,E}$ as a sub-path of $\zeta_\beta$.}
  \label{sub2}
\end{subfigure}
\caption{Diagrams to illustrate $\lambda_\beta$, $\xi_\beta$ and $\xi_{\beta,E}$, as sub-paths of $\zeta_\beta$.}
\label{tevp}
\end{figure}    

\noindent
\textit{Outline of proof}: For any $\beta\in{EVP}^+\setminus{ESP}^+$, we take the corresponding geodesic ray $\zeta_\beta$, starting at the $y$-axis $I$ and terminating at $\beta$ and we take an edge $E'$ in $\mathcal{F}$, which splits $\zeta_\beta$ into a periodic part $\xi_\beta$ and a non-periodic part $\lambda_\beta$. Note that we can view periodicity of $\xi_\beta$ as a geometric property in the sense that when we take the quotient orbifold $\mfaktor{\Gamma_0(m)}{\mathbb{H}}$, $\xi_\beta$ is homotopic to a closed curve. This is true for all $m\in\mathbb{N}$. In particular, any sub-ray of $\zeta_\beta$, which starts after $E'$ will be "geometrically periodic". We show that there exists $a,k\in\mathbb{N}$ such that the line $E$ from $\frac{a}{n^k}$ to $\infty$ intersects $\xi_\beta$. If we take $\xi_{\beta,E}$ to be $\xi_\beta$ with a prefix removed such that $\xi_{\beta,E}$ starts at edge $E$, then $\xi_{\beta,E}$ will also be homotopic to a closed curve in the quotient orbifold $\mfaktor{\Gamma_0(m)}{\mathbb{H}}$. It follows from Corollary \ref{rcc} that $\eta((n^k\xi_{\beta,E},\frac{1}{m}\mathcal{F}))\in{ESP^+}$ for all $m\in\mathbb{N}$.  When we scale $\xi_{\beta,E}$ by $n^k$, we see that $n^k\xi_{\beta,E}$ is still "geometrically periodic" and starts at the line from $a$ to $\infty$. Since $n^k\xi_{\beta,E}\subset{n^k{\zeta_\beta}}$ it follows that $\eta((n^k\zeta_\beta,\mathcal{F}))=a+\eta((n^k\xi_{\beta,E},\mathcal{F}))$ and $\eta((n^k\zeta_\beta,\frac{1}{m}\mathcal{F}))=ma+\eta((n^k\xi_{\beta,E},\frac{1}{m}\mathcal{F}))$ for any $m\in\mathbb{N}$.  See Fig.~\ref{tevp}.

\begin{proof}
We begin by looking at geodesic rays on $\mathbb{H}$. Let $\zeta_\beta$ to be a geodesic ray, starting at the $y$-axis $I$ and terminating at some point $\beta>0$, such that $\beta\in{EVP^+\setminus{ESP^+}}$. We can write $(\zeta_\beta,\mathcal{F})$ to be in the form $L^{b_0}\cdots{R^{a_{2r-1}}}\overline{L^{a_0}\cdots{R^{a_{2s-1}}}}$ with $b_0\in\mathbb{N}\cup\{0\}$ and $a_i,b_j\in\mathbb{N}$. We take $\lambda_\beta$ to be the finite sub-path of $\zeta_\beta$, which starts at $I$, terminates at some edge $E'$ and has cutting sequence $(\lambda_\beta,\mathcal{F})=L^{b_0}\cdots{R^{a_{2r-1}}}$. Similarly, we take $\xi_\beta$ to be the infinite sub-path of $\zeta_\beta$, which starts at $E$, terminates at $\beta$ and has cutting sequence $(\xi_\beta,\mathcal{F})=\overline{L^{a_0}\cdots{R^{a_{2s-1}}}}$. Let $(x,y)$ be the Cartesian co-ordinates of $\zeta_\beta\cap{E'}$. Necessarily $x<\beta$ since the geodesic ray approaches $\beta$ from the left and if we were to assume $x\geq{\beta}$ then the unique geodesic which passes through $(x,y)$ and $\beta$  could not also pass through the line from $0$ to $\infty$. Thus, the interval $[x,\beta)$ is non-empty. By continuity, we can find values $a,k\in\mathbb{N}$ such that $\frac{a}{n^k}\in[x,\beta)$ and the line $E$ from $\frac{a}{n^k}$ to $\infty$, must intersect $\xi_\beta$ (and by extension $\zeta_\beta$). Since $E$ is in $\frac{1}{mn^{k}}\mathcal{F}$ for all $m\in\mathbb{N}$, it follows by rescaling that $mn^{k}\zeta_\beta$ passes through the line from $a$ to $\infty$ for all $m\in\mathbb{N}$.

%We claim that $(\alpha',\frac{1}{mn^{k}}\mathcal{F})_T$ is in $ESP^+$ for all $m\in\mathbb{N}$ and that for $\alpha$ in $ESP^+$ starting at the line from $0$ to $\infty$ with $(\alpha,\frac{1}{n^k}\mathcal{F})=(\alpha',\frac{1}{n^k}\mathcal{F})_T$, we have $(\alpha,\frac{1}{mn^{k}}\mathcal{F})=(\alpha',\frac{1}{mn^{k+l}}\mathcal{F})_T$. 

When we take $\overline{\xi_\beta}$ to be the projection of $\xi_\beta$ in $\mfaktor{\Gamma_0(mn^k)}{\mathbb{H}}$, we see that $\overline{\xi_\beta}$ is homotopic to a closed curve in $\mfaktor{\Gamma_0(mn^k)}{\mathbb{H}}$, since $\eta((\overline{\xi_\beta},\sfaktor{T_{\{1,mn^k\}}}{\sim}))=\eta((\xi_\beta,\mathcal{F}))\in\text{ESP}^+$. Since $\xi_\beta$ intersects the edge $E$ from $\frac{a}{n^k}$ to $\infty$ and $E$ is an edge of $\frac{1}{n^{k}}\mathcal{F}$, $\overline{\xi_\beta}$ intersects $\overline{E}$, the image of $E$ in $\mfaktor{\Gamma_0(mn^k)}{\mathbb{H}}$, infinitely often. If we remove the prefix of $\xi_\beta$ such that it starts from the edge $E$, which we denote ${\xi_{\beta,E}}$, then the projection $\overline{\xi_{\beta,E}}$ in $\mfaktor{\Gamma_0(mn^k)}{\mathbb{H}}$ will also be homotopic to a closed curve (since this is equivalent to just moving the base point of the curve). Then, by Theorem \ref{cc}, $(\overline{\xi_{\beta,E}},\sfaktor{T_{\{n^k,mn^k\}}}{\sim})$ and $(\overline{\xi_{\beta,E}},\sfaktor{T_{\{mn^k,mn^k\}}}{\sim})$ are both essentially periodic. For $m\in\mathbb{N}$, there is a map $\phi:=\begin{psmallmatrix} 1 & -\frac{a}{n^k}\\ 0 & 1\end{psmallmatrix}\in{Isom^+(\frac{1}{mn^{k}}\mathcal{F})}\subset{Isom^+(\frac{1}{n^{k}}\mathcal{F})}$, which maps $E$ to $I$. 
For $\xi_\alpha:=\phi(\xi_{\beta,E})$, since $\phi\in{Isom^+(\frac{1}{mn^{k}}\mathcal{F})}$ for  $m\in\mathbb{N}$, we have $(\xi_{\beta,E},\frac{1}{mn^k}\mathcal{F})_E=\phi((\xi_{\beta,E},\frac{1}{mn^k}\mathcal{F})_E)=(\phi(\xi_{\beta,E}),\phi(\frac{1}{mn^k}))_{\phi({E})}=(\xi_\alpha,\frac{1}{mn^k}\mathcal{F})_I$.

If we take $\alpha$ to be the endpoint of $\xi_\alpha$, we can see that $\alpha+\frac{a}{n^k}=\beta$ and therefore, $mn^k\alpha+ma=mn^k\beta$.  Finally, since $mn^k\beta$ intersects the line from $ma$ to $\infty$, the terms before this line can only affect the first term i.e. $\eta((\zeta_\beta,\frac{1}{mn^k}\mathcal{F}))=ma+\eta((\xi_\beta,\frac{1}{mn^k}\mathcal{F}))_E=ma+\eta((\xi_\alpha,\frac{1}{mn^{k}}\mathcal{F}))$. Here we can take $\alpha'=n^k\alpha$ which will be in $ESP^+$ by Corollary \ref{rcc}. The result follows by relabelling.
\end{proof}

\subsection{Bounds of Eventually Periodic Continued Fractions Grow at least Exponentially}

In this section we give an alternative proof to the statement that every element of $EVP^+$ satisfies pLC \cite{PDs}. We also show that for $\overline{\alpha}\in{ESP^+}$   $\lim_{m\rightarrow\infty} B(m\alpha)=\infty$. Finally, we show that for every $\overline{\alpha}\in{EVP}^+$ there exists  $a,k\in\mathbb{N}$ such that $n^ia\leq{B(n^{i+k}\alpha)}$ for every $i\in\mathbb{N}$. In other words, $B(n^{l}\alpha)$ grows at least exponentially (after some point) for any $\overline{\alpha}\in{EVP}^+$.

\begin{pro}\label{mai}  If $\alpha\in\mathbb{R}_{>0}$ with $\overline{\alpha}\in{ESP^+}$, then $\lim_{m\rightarrow\infty} B(m\alpha)=\infty$.
\end{pro}

\begin{proof} By corollary \ref{rcc}, for every $m\in\mathbb{N}$ we have $\overline{m\alpha}\in{ESP^+}$.  For $k\in\mathbb{N}$ big enough, $k\alpha>1$ and for all $m\geq{k}$, $\overline{m\alpha}$ is of the form $[a^{(m)}_0;\overline{a^{(m)}_1,\ldots,a^{(m)}_{r(m)}}]$ with $0<a^{(m)}_0\leq{a^{(m)}_{r(m)}}$. Here $r(m)$ is the length of the period for $m\alpha$. For each $m\in\mathbb{N}$, $B(m\alpha)\geq{a^{(m)}_{r(m)}}$ by the definition of the function $B(x)$, ${a^{(m)}_{r(m)}}\geq{a^{(m)}_0}$ by the definition of essentially periodic continued fractions and ${a^{(m)}_0=\lfloor{m\alpha}\rfloor}$ by the construction of continued fractions. In particular, $B(m\alpha)\geq{a^{(m)}_{r(m)}}\geq{a^{(m)}_0=\lfloor{m\alpha}\rfloor}$.  Since $\lfloor{m\alpha}\rfloor\rightarrow\infty$ for every $\alpha\in\mathbb{R}_{>0}$, it follows that $B({m\alpha})\rightarrow\infty$.
\end{proof}

\begin{pro}\label{evpplc} If $\alpha\in\mathbb{R}_{>0}$ with $\overline{\alpha}\in{EVP^+}$, then $\lim_{i\rightarrow\infty} B(n^i\alpha)=\infty$. In particular, there exists  $a,k\in\mathbb{N}$ such that $n^ia\leq{B(n^{i+k}\alpha)}$ for every $i\in\mathbb{N}$. Every element in $EVP^+$ satisfies pLC.
\end{pro}

\begin{proof}
By Theorem \ref{evp}, for $\beta\in{EVP^+}$ we can find some $k'\in\mathbb{N}\cup\{0\}$, such that $n^{k'}\beta$ behaves like some element $n^{k'}\alpha$ of $ESP^+$. In particular, by Proposition \ref{mai}, for every $\beta\in{EVP^+}$, $\lim_{i\rightarrow\infty} B(n^{k'+i}\beta)=\infty$ and by taking $n=p$ prime, pLC follows. We can take $k\in\mathbb{N}$, with $k\geq{k'}$, such that $\lfloor{n^k\alpha\rfloor}>1$. We know that $B(n^{k'+i}\alpha)=B(n^{k'+i}\beta)$, since $n^{k'}\beta$ and $n^{k'}\alpha$ only differ by their first term (by Theorem \ref{evp}). Therefore, $B(n^{i+k}\beta)=B(n^{i+k}\alpha)\geq\lfloor{n^{i+k}\alpha}\rfloor\geq{\lfloor{n^i\lfloor{n^k\alpha}}\rfloor}\rfloor={n^i\lfloor{n^k\alpha}\rfloor}=n^ia^{(k)}_0$.
\end{proof}

\end{document}